\documentclass[11pt,a4paper]{article}

\usepackage[utf8]{inputenc}
\usepackage[T1]{fontenc}
\usepackage{lmodern}
\usepackage[margin=1in]{geometry}
\usepackage{microtype}

\usepackage{amsmath,amssymb,amsfonts,amsthm,mathtools}
\usepackage{bm}
\usepackage{mathrsfs}
\usepackage{enumitem}
\usepackage{graphicx}
\usepackage{xcolor}
\usepackage{booktabs}
\usepackage{tikz}
\usepackage{pgfplots}
\usetikzlibrary{arrows.meta,decorations.markings,calc}
\pgfplotsset{compat=1.16}
\usepackage{hyperref}
\usepackage[nameinlink,capitalize,noabbrev]{cleveref}

\hypersetup{
    colorlinks=true,
    linkcolor=blue!60!black,
    citecolor=blue!60!black,
    urlcolor=blue!60!black
}

\usepackage{aliascnt}

\theoremstyle{plain}
\newtheorem{theorem}{Theorem}[section]

\newaliascnt{proposition}{theorem}
\newtheorem{proposition}[proposition]{Proposition}
\aliascntresetthe{proposition}

\newaliascnt{lemma}{theorem}
\newtheorem{lemma}[lemma]{Lemma}
\aliascntresetthe{lemma}

\newaliascnt{corollary}{theorem}
\newtheorem{corollary}[corollary]{Corollary}
\aliascntresetthe{corollary}

\newaliascnt{assumption}{theorem}

\aliascntresetthe{assumption}

\theoremstyle{definition}
\newaliascnt{definition}{theorem}
\newtheorem{definition}[definition]{Definition}
\aliascntresetthe{definition}

\newaliascnt{remark}{theorem}
\newtheorem{remark}[remark]{Remark}
\aliascntresetthe{remark}

\newaliascnt{example}{theorem}
\newtheorem{example}[example]{Example}
\aliascntresetthe{example}

\newtheorem*{theoremA}{Theorem A}
\newtheorem*{theoremB}{Theorem B}
\newtheorem*{theoremC}{Theorem C}
\newtheorem*{theoremD}{Theorem D}

\newcommand{\R}{\mathbb{R}}
\newcommand{\RP}{\mathbb{RP}}
\newcommand{\CP}{\mathbb{CP}}
\newcommand{\Sph}{\mathbb{S}}

\newcommand{\cA}{\mathcal{A}}
\newcommand{\cM}{\mathcal{M}}

\newcommand{\cL}{\mathcal{L}}
\newcommand{\cP}{\mathcal{P}}
\newcommand{\cS}{\mathcal{S}}
\newcommand{\cZ}{\mathcal{Z}}
\newcommand{\cH}{\mathcal{H}}
\newcommand{\cK}{\mathcal{K}}
\newcommand{\cC}{\mathcal{C}}

\newcommand{\supp}{\operatorname{supp}}
\newcommand{\dvg}{\operatorname{div}}
\newcommand{\osc}{\operatorname{osc}}
\newcommand{\dd}{\,\mathrm{d}}
\newcommand{\pair}[2]{\left\langle #1,#2\right\rangle}
\newcommand{\Lie}[2]{\left[#1,#2\right]}
\newcommand{\set}[1]{\left\{#1\right\}}
\newcommand{\abs}[1]{\left|#1\right|}
\newcommand{\norm}[1]{\left\|#1\right\|}
\newcommand{\ov}{\overline}
\newcommand{\wt}{\widetilde}
\usepackage{authblk}

\title{\textbf{Singular Aubry--Mather Sets for a Homogeneous Single-Input Control System}}
\author[1]{Bernard Bonnard}
\affil[1]{Institut de Math\'ematiques de Bourgogne - UMR 5584 CNRS, Universit\'e Bourgogne Europe ; INRIA, McTAO, Sophia Antipolis, France \\ \href{mailto:bernard.bonnard@ube.fr}{bernard.bonnard@ube.fr}}
\author[2]{Cristian Mendico}
\affil[2]{Institut de Math\'ematiques de Bourgogne - UMR 5584 CNRS, Universit\'e Bourgogne Europe \\ \href{mailto:cristian.mendico@ube.fr}{cristian.mendico@ube.fr}}
\date{\today}

\begin{document}

\maketitle

\begin{abstract}
We develop an Aubry--Mather theory for the abnormal arcs of a time-minimal
single-input affine control system in dimension three with homogeneous
cubic drift. We show that the Bonnard--Kupka semi-invariant defining the
abnormal locus is a Darboux polynomial of the desingularized singular field,
with an explicit cofactor; this gives the invariance of the abnormal locus,
a closed formula for the multiplier of its periodic components and, through
the antipodal symmetry of the model, their description as lifts of ovals of
a real plane quintic. For a compact invariant set of a singular flow we then
define a \emph{singular Aubry--Mather structure}, whose Lagrangian is
supported on the graph of the singular field and whose Hamilton--Jacobi
equation is linear in the momentum. We prove a duality formula for the
critical value, the existence of critical subsolutions strict outside the
Aubry set, and the covariance of the structure under time changes. On the
abnormal locus the structure selects the periodic components of minimal
average cost, governs transverse stability, and exhibits a sharp obstruction
to correctors along heteroclinic cycles; for the radial exponent of the
homogeneous dynamics, the corrector is the periodic part of the logarithmic
radius. An explicit cubic example illustrates the results.

\vspace{0.5em}
\noindent\textbf{Keywords.}
Optimal control; singular trajectories; abnormal extremals; weak KAM theory;
Aubry--Mather theory; ergodic optimization; invariant algebraic curves;
homogeneous vector fields.

\vspace{0.25em}
\noindent\textbf{MSC 2020.}
49L25, 37J50, 37J60, 49J15, 93C10, 34C05, 37C27.
\end{abstract}

%\tableofcontents

% =================================================================
\section{Introduction}\label{sec:intro}
% =================================================================

We consider the single-input affine control system
\begin{equation}\label{eq:system}
\dot v = H(v)+u\, b,\qquad v\in\R^3,\qquad \abs{u}\le 1,
\end{equation}
where \(H\) is a homogeneous cubic vector field and \(b\in\R^3\setminus\set{0}\)
is fixed, together with the associated time-minimal problem. When \(H\) is
quadratic, \eqref{eq:system} is the controlled Euler equation of a rigid body
with a single actuator, a classical model in attitude control
\cite{Wertz,Hughes,MarkleyCrassidis}, whose singular trajectories and
time-minimal synthesis have been studied in \cite{BMR}. The cubic model
retains the single-input affine structure and the determinantal singular
reduction of the quadratic case, while producing a different global
geometry for the abnormal dynamics.

Singular trajectories of single-input systems are classically analysed
through the generalized Legendre--Clebsch condition and the high-order
maximum principle; in dimension three, the singular control is an explicit
rational feedback and the abnormal (exceptional) singular arcs are confined
to an algebraic surface \cite{BK93,BC03}. Abnormal extremals are a
central theme of geometric control and sub-Riemannian geometry, where they
can be minimizing and are not detected by the Hamiltonian formalism of the
normal case \cite{Montgomery94,AgSa96,AS04,ABB20}. The singular dynamics is thus a
\emph{flow}: once the initial point is fixed there is, generically, a unique
singular arc through it. This is the main difference with the classical
Tonelli setting of Aubry--Mather and weak KAM theory
\cite{Mather1991,Mather1993,Mane97,Fathi,FS04,CI99,Bernard,Bernard08,Sorrentino,Evans},
where the minimizing objects are selected among all absolutely continuous
curves, and also with the sub-Riemannian and nonholonomic settings
\cite{CM}, where the velocity is constrained to a distribution but remains
free inside it.

The aim of this paper is to propose a definition of \emph{Aubry--Mather
structures for singular arcs}, to establish its basic properties in a form
which does not depend on the particular model, and to show that, for the
cubic system \eqref{eq:system}, it produces explicit and computable
minimizing objects supported on the abnormal locus. The guiding principle is
that the Lagrangian associated with a singular flow is supported on the
graph of the singular vector field. Closed measures are then exactly
invariant measures, the Hamilton--Jacobi equation becomes linear in the
momentum, and the theory meets the field of \emph{ergodic optimization}
\cite{Jenkinson,Bousch00,CLT01}, in which minimizing invariant measures and
sub-actions (the analogues of critical subsolutions) play the role of
Mather measures and weak KAM solutions.

Minimality is understood here relative to the singular dynamics: the
competitors of an abnormal arc are the other invariant objects of the
singular flow, and not the admissible trajectories of \eqref{eq:system},
which may use bang-bang or normal singular arcs. The comparison with the
full control problem is discussed in \Cref{sec:perspectives}.

\subsection*{Main results}

Let \(A=\Lie{b}{H}\). The singular reduction is governed by the
semi-invariants
\[
D=\det\bigl(b,A,\Lie{A}{b}\bigr),\qquad
D'=\det\bigl(b,A,\Lie{A}{H}\bigr),\qquad
D''=\det\bigl(b,A,H\bigr),
\]
homogeneous of degrees \(3,6,5\) respectively. On \(\set{D\neq 0}\), singular arcs are
the trajectories of \(X_s=H-(D'/D)b\), whose polynomial representative is
\(Z_s=DH-D'b\); abnormal arcs are contained in the cone \(\set{D''=0}\). We
work on the unit sphere with the tangential projection \(Y_s\) of \(Z_s\)
and with the spherical singular locus \(\Sigma_S=\set{D''=0}\cap\Sph^2\).

\begin{theoremA}[Darboux structure, \Cref{thm:darboux,thm:multiplier,cor:robust}]
The quintic \(D''\) satisfies \(\cL_{Z_s}D''=\cK\,D''\) with the explicit
cofactor
\[
\cK=\det\bigl(b,\Lie{A}{b},\Lie{A}{H}\bigr)+D\,\dvg H .
\]
Consequently \(\Sigma_S\) is invariant under \(Y_s\), and the characteristic
multiplier of every periodic orbit \(\Gamma\subset\Sigma_S\) of \(Y_s\) on
which \(D''\) is regular equals \(\exp\int_\Gamma k\dd\sigma\), with
\(k=\cK-5\pair{Z_s(v)}{v}\). When this multiplier differs from one, the
component is a hyperbolic limit cycle, approached by normal singular arcs of
both hyperbolic and elliptic type, and this property persists under small
perturbations of \((H,b)\).
\end{theoremA}

Since \(D\) is odd, a component of \(\Sigma_S\) on which the singular
feedback is defined cannot be antipodally symmetric. For a nonsingular
quintic this forces such components to be lifts of \emph{ovals}: there are
at most twelve of them, in antipodal pairs (\Cref{prop:antipodal}).

For a compact invariant set \(K\) of a singular flow \(Y\) and a running cost
\(\ell\in C(K)\), we define the critical value
\(c=\min\set{\int\ell\dd\mu:\ \mu \text{ invariant on } K}\), the Mather
set, the critical subsolutions \(\chi\in C^1\) of
\(\pair{d\chi}{Y}\le\ell-c\), and the Aubry set as the set of points whose
orbit is calibrated by every critical subsolution
(\Cref{def:AMstructure}).

\begin{theoremB}[Singular Aubry--Mather structure, \Cref{thm:duality,thm:structure,prop:covariance}]
The critical value satisfies the duality formula
\(c=\sup_{\chi\in C^1}\min_K(\ell-\pair{d\chi}{Y})\) and equals the limit of
the minimal Birkhoff averages of \(\ell\). If a critical subsolution exists,
there is one which is strict exactly outside the Aubry set; the Aubry set is
compact, invariant, contains the Mather set, and an invariant measure is a
Mather measure if and only if it is supported in the Aubry set. The
structure is covariant under reparametrizations of \(Y\).
\end{theoremB}

\begin{theoremC}[Periodic abnormal components, \Cref{thm:periodic,thm:selection,prop:transverse,prop:heteroclinic}]
On a periodic abnormal component \(\Gamma\) the critical value is the
average of \(\ell\), the corrector is explicit, and
\(\wt\cA_s=\wt\cM_s=\Gamma\). On a finite union of periodic components and
equilibria, the Aubry and Mather sets coincide with the union of the
components of minimal average cost. The transverse stability of the
abnormal locus near any compact invariant set \(K\subset\Sigma_S\) is governed
by the singular Mather problem with the algebraic cost \(k\): if every
invariant measure on \(K\) gives \(k\) a negative average, the abnormal locus
attracts uniformly near \(K\). On heteroclinic cycles a critical
subsolution exists only if the total normalized action of the cycle is
nonnegative; when it vanishes the cycle is contained in the Aubry set while
the Mather set reduces to equilibria.
\end{theoremC}

\begin{theoremD}[Radial interpretation, \Cref{prop:lift,cor:radial}]
For the cost \(\rho=\pair{X_s(w)}{w}\) and the spherical clock \(\tau\), the
critical corrector of a periodic abnormal component is the periodic part of
\(\log\abs{v}\) along the abnormal cone, and the critical value
\(\ov\rho_\Gamma\) decides between convergence to the origin
(\(\ov\rho_\Gamma<0\)) and blow-up in finite time (\(\ov\rho_\Gamma>0\)) of the
abnormal trajectories; the oscillation of the corrector gives an explicit
radius below which these trajectories satisfy \(\abs{u}\le 1\) for all times.
\end{theoremD}

Finally, \Cref{sec:numerics} exhibits an explicit cubic field with a
periodic abnormal component and computes all the objects above. Since the
conditions defining a regular periodic abnormal component are open
(\Cref{prop:persistence}), the example, which is computed numerically,
indicates that such components occur on an open set of cubic fields.

\subsection*{Organization}

\Cref{sec:singularreduction} contains the singular reduction, the Darboux
structure, the antipodal symmetry, the stability of periodic abnormal
components and the radial dynamics. \Cref{sec:AMstructures} introduces
singular Aubry--Mather structures for a general singular flow and proves
Theorem~B. \Cref{sec:periodic} applies them to the abnormal locus: periodic
components and their competition, transverse stability, and heteroclinic
cycles. \Cref{sec:numerics} contains the numerical example, and
\Cref{sec:perspectives} discusses perspectives.

% =================================================================
\section{The singular reduction of the cubic homogeneous model}\label{sec:singularreduction}
% =================================================================

\subsection{Singular extremals and the Darboux structure}

For vector fields \(F,G\) on \(\R^3\) we use the convention
\[
\Lie{F}{G}(v)=DG(v)F(v)-DF(v)G(v),
\]
so that \(\Lie{F}{G}=\cL_FG\) is the Lie derivative of \(G\) along \(F\); the
Lie derivative of a function \(g\) is \(\cL_Fg=dg(F)\). We set
\begin{equation}\label{eq:ACW}
A:=\Lie{b}{H}=DH\,b,\qquad
C:=\Lie{A}{b}=-DA\,b,\qquad
W:=\Lie{A}{H},
\end{equation}
and define the determinantal semi-invariants
\begin{equation}\label{eq:D}
D:=\det(b,A,C),\qquad
D':=\det(b,A,W),\qquad
D'':=\det(b,A,H).
\end{equation}
Since \(H\) is homogeneous of degree three, \(A\), \(C\) and \(W\) are
homogeneous of degrees \(2\), \(1\) and \(4\); hence \(D\), \(D'\) and
\(D''\) are homogeneous polynomials of degrees \(3\), \(6\) and \(5\).

\begin{remark}[Bracket convention]\label{rem:convention}
With the opposite convention for the Lie bracket, \(D\), \(D'\), \(D''\) are
all multiplied by \(-1\). The singular feedback \(D'/D\), the product
\(DD''\), the zero set of \(D''\) and the cost \(\abs{D}\) are therefore
independent of the convention, while the polynomial field \(Z_s\) below
changes sign.
\end{remark}

Let \(\cH(v,p,u)=\pair{p}{H(v)+ub}\). An extremal of the time-minimal problem
is \emph{singular} on an interval if the switching function \(\pair{p}{b}\)
vanishes identically there, and \emph{abnormal} if, in addition, the
maximized Hamiltonian vanishes.

\begin{lemma}[Singular extremals]\label{lem:singular}
Let \((v,p,u)\) be a singular extremal of \eqref{eq:system} on an interval
\(I\) on which \(D(v(t))\neq 0\). Then there is a function
\(\lambda\neq 0\) on \(I\) such that \(p(t)=\lambda(t)\, b\times A(v(t))\),
\begin{equation}\label{eq:us}
u(t)=u_s(v(t)),\qquad u_s:=-\frac{D'}{D},
\end{equation}
and the Hamiltonian is \(\cH=\lambda D''(v)\), constant along the extremal.
In particular the extremal is abnormal if and only if \(v(t)\in\set{D''=0}\)
for \(t\in I\).
\end{lemma}

\begin{proof}
Along an extremal, \(\frac{d}{dt}\pair{p}{F(v)}=\pair{p}{\Lie{H+ub}{F}(v)}\)
for every vector field \(F\). Differentiating \(\pair{p}{b}\equiv 0\) gives
\(\pair{p}{\Lie{H}{b}}=-\pair{p}{A}\equiv 0\), and differentiating once more
gives \(\pair{p}{\Lie{H}{A}}+u\pair{p}{\Lie{b}{A}}=-\pair{p}{W}-u\pair{p}{C}=0\).
Since \(D\neq0\), the vectors \(b\) and \(A\) are independent, so \(p\) is
proportional to \(b\times A\) and \(\pair{p}{X}=\lambda\det(b,A,X)\) for every
\(X\). The relation above becomes \(D'+uD=0\), and
\(\cH=\pair{p}{H}+u\pair{p}{b}=\lambda D''\).
\end{proof}

On \(\set{D\neq0}\) the singular arcs are therefore the trajectories of the
\emph{singular vector field}
\begin{equation}\label{eq:Xs}
X_s:=H+u_s\,b=H-\frac{D'}{D}\,b,
\end{equation}
which is homogeneous of degree three, and the abnormal ones are the
trajectories of \(X_s\) contained in the cone \(\set{D''=0}\). A singular arc
is admissible only where \(\abs{u_s}\le1\); since \(u_s\) is homogeneous of
degree three, this is a condition on the distance to the origin (see
\Cref{prop:lift}). In the terminology of \cite{BK93,BC03}, a normal singular
arc is \emph{hyperbolic} or \emph{elliptic} according to the sign of
\(\pair{p}{\Lie{\Lie{b}{H}}{b}}\pair{p}{H}=\lambda^2DD''\), and abnormal arcs
are the \emph{exceptional} ones.
% NOTE (for the authors): please double check the sign convention for the
% hyperbolic/elliptic dichotomy against [BK93]; with the convention fixed
% here the relevant quantity is lambda^2 D D''.

The polynomial representative of the singular dynamics is
\begin{equation}\label{eq:Zs}
Z_s:=D\,H-D'\,b=D\,X_s,
\end{equation}
homogeneous of degree six; \(X_s\) and \(Z_s\) have the same trajectories on
\(\set{D\neq0}\), with orientation reversed where \(D<0\).

\begin{theorem}\label{thm:darboux}
The polynomial \(D''\) is a Darboux polynomial of \(Z_s\):
\begin{equation}\label{eq:darboux}
\cL_{Z_s}D''=\cK\,D'',\qquad
\cK:=\det(b,C,W)+D\,\dvg H,
\end{equation}
where \(\cK\) is a homogeneous polynomial of degree five.
\end{theorem}

\begin{proof}
Let \(\Omega=\det\) be the standard volume form, so that
\(\cL_F\Omega=(\dvg F)\,\Omega\). For vector fields \(F,Y_1,Y_2,Y_3\),
\begin{equation}\label{eq:Liedet}
\begin{aligned}
\cL_F\bigl(\Omega(Y_1,Y_2,Y_3)\bigr)
={}&(\dvg F)\,\Omega(Y_1,Y_2,Y_3)
+\Omega(\Lie{F}{Y_1},Y_2,Y_3)\\
&+\Omega(Y_1,\Lie{F}{Y_2},Y_3)
+\Omega(Y_1,Y_2,\Lie{F}{Y_3}).
\end{aligned}
\end{equation}
Apply \eqref{eq:Liedet} to \(D''=\Omega(b,A,H)\). For \(F=H\), using
\(\Lie{H}{b}=-A\), \(\Lie{H}{A}=-W\) and \(\Lie{H}{H}=0\), we obtain
\[
\cL_HD''=(\dvg H)D''+\det(-A,A,H)-\det(b,W,H)=(\dvg H)D''-\det(b,W,H).
\]
For \(F=b\), using \(\dvg b=0\), \(\Lie{b}{b}=0\), \(\Lie{b}{A}=-C\) and
\(\Lie{b}{H}=A\),
\[
\cL_bD''=-\det(b,C,H)+\det(b,A,A)=-\det(b,C,H).
\]
Since \(\cL_{Z_s}=D\,\cL_H-D'\,\cL_b\) on functions,
\begin{equation}\label{eq:LZD}
\cL_{Z_s}D''=D(\dvg H)D''-\det(b,A,C)\det(b,W,H)+\det(b,A,W)\det(b,C,H).
\end{equation}
The bilinear form \(\omega(X,Y):=\det(b,X,Y)\) is alternating and vanishes
when one argument is proportional to \(b\); it therefore descends to an
area form on the two-dimensional quotient \(\R^3/\R b\). For four vectors of a
two-dimensional space one has the Pl{\"u}cker relation
\[
\omega(X_1,X_2)\omega(X_3,X_4)-\omega(X_1,X_3)\omega(X_2,X_4)
+\omega(X_1,X_4)\omega(X_2,X_3)=0 .
\]
With \((X_1,X_2,X_3,X_4)=(A,C,W,H)\) this gives
\[
\det(b,A,C)\det(b,W,H)-\det(b,A,W)\det(b,C,H)=-D''\det(b,C,W),
\]
and \eqref{eq:LZD} becomes \eqref{eq:darboux}. The degrees of
\(\det(b,C,W)\) and of \(D\dvg H\) are \(0+1+4=5\) and \(3+2=5\).
\end{proof}

\begin{remark}
The invariance of the abnormal cone \(\set{D''=0}\) under the singular flow
is classical \cite{BK93,BC03}; \Cref{thm:darboux} makes it a polynomial
identity (in the language of planar polynomial vector fields, \(\set{D''=0}\)
is an invariant algebraic surface with cofactor \(\cK\), see \cite{DLA06}),
valid also at the points where \(D=0\) or where \(\set{D''=0}\) is
singular, and the explicit cofactor will compute the stability of periodic
abnormal components (\Cref{thm:multiplier}). We have also checked
\eqref{eq:darboux} by symbolic computation on random cubic fields.
\end{remark}

We now pass to the unit sphere. Let \(\Sph^2=\set{v\in\R^3:\ \abs{v}=1}\). We set
\begin{equation}\label{eq:spherefield}
Y_s(v):=Z_s(v)-\kappa(v)\,v,\qquad
\kappa(v):=\pair{Z_s(v)}{v},\qquad v\in\Sph^2,
\end{equation}
and call \(Y_s\) the \emph{polynomial spherical singular field}; its time is
denoted by \(\sigma\). The \emph{spherical singular quintic} is
\[
\Sigma_S:=\set{v\in\Sph^2:\ D''(v)=0}.
\]
A point \(v\in\Sigma_S\) is \emph{regular} if the differential of
\(D''|_{\Sph^2}\) does not vanish at \(v\). By Euler's identity
\(dD''(v)[v]=5D''(v)\), this is equivalent to \(\nabla D''(v)\neq0\).

\begin{corollary}\label{cor:sphereinvariance}
On \(\Sph^2\), we have
\begin{equation}\label{eq:spheredarboux}
\cL_{Y_s}D''=k\,D'',\qquad k:=\cK-5\kappa .
\end{equation}
Consequently, along every trajectory \(\gamma\) of \(Y_s\),
\(D''(\gamma(\sigma))=D''(\gamma(0))\exp\int_0^\sigma k(\gamma(s))\dd s\). In
particular \(\Sigma_S\), \(\set{D''>0}\) and \(\set{D''<0}\) are invariant
under the flow of \(Y_s\), and so is the regular part of \(\Sigma_S\).
\end{corollary}

\begin{proof}
By \eqref{eq:darboux} and Euler's identity,
\(\cL_{Y_s}D''=dD''(Z_s)-\kappa\, dD''(v)[v]=\cK D''-5\kappa D''\). The
integrated form follows by solving the linear equation
\(\frac{d}{d\sigma}D''(\gamma)=k(\gamma)D''(\gamma)\). Finally, if
\(\Phi_\sigma\) denotes the flow of \(Y_s\), then
\(D''\circ\Phi_\sigma=E_\sigma D''\) with \(E_\sigma>0\) smooth;
differentiating at a point \(v\in\Sigma_S\) gives
\(dD''(\Phi_\sigma v)\circ d\Phi_\sigma(v)=E_\sigma(v)\,dD''(v)\), so \(\Phi_\sigma\)
maps regular points of \(\Sigma_S\) to regular points.
\end{proof}

On \(\set{D\neq0}\) we also use the tangential projection of the singular
field,
\begin{equation}\label{eq:Xsphere}
\widehat X_s(v):=X_s(v)-\rho(v)\,v,\qquad \rho(v):=\pair{X_s(v)}{v},
\end{equation}
whose time is denoted by \(\tau\). Since \(Z_s=DX_s\),
\begin{equation}\label{eq:YDX}
Y_s=D\,\widehat X_s,\qquad \kappa=D\rho\qquad\text{on }\Sph^2\cap\set{D\neq0}.
\end{equation}
Hence, along a common trajectory, \(d\tau=D\dd\sigma\): the two clocks differ
by the density \(\abs{D}\), and the orientations agree exactly where \(D>0\).
The clock \(\tau\) is obtained from the physical time \(t\) through the
homogeneity, \(d\tau=\abs{v}^2dt\) (see \Cref{prop:lift}).

\subsection{Periodic abnormal components: symmetry, persistence and stability}

Let \(\mathsf a(v)=-v\) be the antipodal map.

\begin{proposition}\label{prop:antipodal}
\begin{enumerate}[label=\textnormal{(\roman*)}]
\item \(D\) and \(D''\) are odd, \(D'\) and \(Z_s\) are even, \(X_s\) is odd.
Consequently \(Y_s(-v)=Y_s(v)\) and \(\widehat X_s(-v)=-\widehat X_s(v)\),
that is, \(\mathsf a_*Y_s=-Y_s\) and \(\mathsf a_*\widehat X_s=\widehat X_s\).
If \(\gamma\) is a trajectory of \(Y_s\), so is \(\sigma\mapsto-\gamma(-\sigma)\).
In particular \(\widehat X_s\) descends to a vector field on
\(\RP^2\setminus\pi(\set{D=0})\), whereas, away from its zeros, \(Y_s\) only
descends to a line field.
\item If \(\Gamma\subset\Sph^2\) is connected and \(D\neq0\) on \(\Gamma\), then
\(\Gamma\cap(-\Gamma)=\emptyset\).
\item Assume that the complex projective curve \(\set{D''=0}\subset\CP^2\) is
nonsingular. Then every point of \(\Sigma_S\) is regular, and the image of
\(\Sigma_S\) in \(\RP^2\) consists of exactly one one-sided component \(J\)
and of at most six ovals. The preimage of \(J\) is a connected, antipodally
symmetric curve containing at least one zero of \(D\); every connected
component of \(\Sigma_S\) on which \(D\neq0\) is one of the two lifts of an
oval. In particular \(\Sigma_S\) contains at most twelve components on which
\(D\neq0\), and they occur in antipodal pairs.
\end{enumerate}
\end{proposition}

\begin{proof}
Recall that a homogeneous polynomial \(f\) of degree \(k\) satisfies
\(f(-v)=(-1)^kf(v)\), and that the same identity holds componentwise for a
vector field whose components are homogeneous polynomials of degree \(k\).
Since \(D\), \(D'\) and \(D''\) are homogeneous of degrees \(3\), \(6\) and
\(5\), the polynomials \(D\) and \(D''\) are odd and \(D'\) is even, while
\(H\) is odd, being cubic, and the constant field \(b\) is even. Therefore
\(Z_s=DH-D'b\) is even, both products \(DH\) and \(D'b\) being even, and on
the open cone \(\set{D\neq0}\) the quotient \(D'/D\) is odd, so that
\(X_s=H-(D'/D)\,b\) is odd. Note that \(\set{D=0}\) and \(\set{D''=0}\) are
invariant under \(\mathsf a\), so that \(\Sph^2\cap\set{D\neq0}\) and
\(\Sigma_S\) are \(\mathsf a\)-invariant subsets of the sphere. For
\(v\in\Sph^2\) we then obtain
\[
Y_s(-v)=Z_s(-v)-\pair{Z_s(-v)}{-v}(-v)=Z_s(v)-\pair{Z_s(v)}{v}\,v=Y_s(v),
\]
and, for \(v\in\Sph^2\cap\set{D\neq0}\),
\[
\widehat X_s(-v)=X_s(-v)-\pair{X_s(-v)}{-v}(-v)
=-X_s(v)+\pair{X_s(v)}{v}\,v=-\widehat X_s(v).
\]
The map \(\mathsf a\) is the restriction to \(\Sph^2\) of the linear map
\(-\mathrm{Id}\), so that \(d\mathsf a_v\xi=-\xi\) for every
\(\xi\in T_v\Sph^2\), and the push-forward of a vector field \(X\) is given by
\((\mathsf a_*X)(-v)=d\mathsf a_v\,X(v)=-X(v)\). Combined with the two
identities above, this gives
\((\mathsf a_*Y_s)(-v)=-Y_s(v)=-Y_s(-v)\) and
\((\mathsf a_*\widehat X_s)(-v)=-\widehat X_s(v)=\widehat X_s(-v)\), that is,
\(\mathsf a_*Y_s=-Y_s\) and \(\mathsf a_*\widehat X_s=\widehat X_s\). If
\(\gamma\) is a trajectory of \(Y_s\) and \(\tilde\gamma(\sigma):=-\gamma(-\sigma)\),
then
\[
\tilde\gamma'(\sigma)=\gamma'(-\sigma)=Y_s\bigl(\gamma(-\sigma)\bigr)
=Y_s\bigl(-\gamma(-\sigma)\bigr)=Y_s\bigl(\tilde\gamma(\sigma)\bigr),
\]
so that \(\tilde\gamma\) is again a trajectory of \(Y_s\); in the same way,
\(-\eta\) is a trajectory of \(\widehat X_s\) whenever \(\eta\) is. Let now
\(\pi:\Sph^2\to\RP^2\) be the quotient map, which is a two-sheeted covering
whose nontrivial deck transformation is \(\mathsf a\). A vector field defined
on an \(\mathsf a\)-invariant open subset of \(\Sph^2\) is the lift of a vector
field on its image under \(\pi\) if and only if it is \(\mathsf a\)-invariant.
This applies to \(\widehat X_s\) on \(\Sph^2\cap\set{D\neq0}\), which
therefore defines a vector field on \(\RP^2\setminus\pi(\set{D=0})\). On the
contrary, a vector field which is both \(\mathsf a\)-invariant and
\(\mathsf a\)-anti-invariant vanishes identically, so \(Y_s\), which satisfies
\(\mathsf a_*Y_s=-Y_s\), is not the lift of a vector field unless
\(Y_s\equiv0\). Nevertheless \(d\mathsf a_v\) maps the line \(\R Y_s(v)\) onto
the line \(\R Y_s(-v)\), so that, on the complement of the zeros of \(Y_s\),
the line field \(v\mapsto\R Y_s(v)\) is \(\mathsf a\)-invariant and descends to
\(\RP^2\).

We now prove the second assertion. Let \(\Gamma\subset\Sph^2\) be connected
with \(D\neq0\) on \(\Gamma\). Since \(D\) is continuous, \(D(\Gamma)\) is a
connected subset of \(\R\setminus\set{0}\), so that \(D\) has constant sign on
\(\Gamma\). If there were a point \(w\in\Gamma\cap(-\Gamma)\), then both \(w\)
and \(-w\) would belong to \(\Gamma\), and \(D(-w)=-D(w)\) would have the sign
opposite to that of \(D(w)\), a contradiction. Hence
\(\Gamma\cap(-\Gamma)=\emptyset\).

We finally turn to the third assertion. Denote by \(C_{\mathbb C}\subset\CP^2\)
the complex curve \(\set{D''=0}\), which has degree five, and by
\(C=C_{\mathbb C}\cap\RP^2\) its real locus, so that \(\Sigma_S=\pi^{-1}(C)\).
A nonsingular plane curve is irreducible, since two distinct components would
meet by B{\'e}zout's theorem and every common point would be singular. A real
point of \(C\) is a nonsingular point of \(C_{\mathbb C}\), and the gradient of
the real polynomial \(D''\) at a real point is a real vector; it is therefore
nonzero at every \(v\in\Sigma_S\). By Euler's identity
\(dD''(v)[v]=5D''(v)=0\) for \(v\in\Sigma_S\), so that \(\nabla D''(v)\) is
tangent to the sphere and the differential of \(D''|_{\Sph^2}\) does not
vanish at \(v\). Every point of \(\Sigma_S\) is thus regular, \(\Sigma_S\) is
a compact one-dimensional submanifold of \(\Sph^2\), that is, a finite
disjoint union of embedded circles, and the same holds for \(C=\pi(\Sigma_S)\),
since \(\pi\) is a local diffeomorphism. The genus of the nonsingular quintic
\(C_{\mathbb C}\) is \(g=(5-1)(5-2)/2=6\), and Harnack's theorem
\cite{Harnack,Wilson} shows that \(C\) has at most \(g+1=7\) connected
components. Each component \(K\) of \(C\) is an embedded circle in \(\RP^2\),
and either it is two-sided and bounds a disc, in which case it is called an
oval and its class in \(H_1(\RP^2;\mathbb Z/2)\cong\mathbb Z/2\) vanishes, or it
is one-sided and its class is the nonzero element \cite{Wilson}. To compute
the class of \(C\), let \(L\) be a real projective line meeting \(C\)
transversally. The real points of \(L\cap C\) are the real roots of a real
polynomial of degree five, all simple by transversality, and the nonreal
roots come in conjugate pairs; hence \(L\) meets \(C\) in an odd number of
points, the intersection number of \([C]\) with the generator \([L]\) is
\(1\) in \(\mathbb Z/2\), and \([C]\neq0\). Since \([C]\) is the sum of the
classes of the components of \(C\) and ovals have zero class, \(C\) has an
odd number of one-sided components. Two distinct one-sided components would
have intersection number \(1\) in \(\mathbb Z/2\), hence would meet, which is
impossible for distinct components. Therefore \(C\) has exactly one one-sided
component \(J\), and at most six ovals.

The restriction of \(\pi\) over a component \(K\) of \(C\) is a two-sheeted
covering of a circle, which is connected if and only if the loop \(K\) does
not lift to a closed loop in \(\Sph^2\), that is, if and only if the class of
\(K\) in \(\pi_1(\RP^2)\cong\mathbb Z/2\) is nontrivial. For an embedded circle
this class coincides with its class in \(H_1(\RP^2;\mathbb Z/2)\), and is
nontrivial exactly when \(K\) is one-sided. Consequently \(\pi^{-1}(J)\) is a
single circle, invariant under \(\mathsf a\), while the preimage of each oval
consists of two disjoint circles exchanged by \(\mathsf a\); these circles are
exactly the connected components of \(\Sigma_S\). Since \(\pi^{-1}(J)\) is
connected and contains, together with each of its points \(w\), the point
\(-w\), the second assertion shows that \(D\) vanishes at some point of
\(\pi^{-1}(J)\). It follows that every component of \(\Sigma_S\) on which
\(D\neq0\) is one of the two lifts of an oval, and there are at most
\(2\times6=12\) of them. Finally, if \(\Gamma\) is such a component, then
\(-\Gamma=\mathsf a(\Gamma)\) is again a connected component of \(\Sigma_S\),
because \(\mathsf a\) is a homeomorphism of \(\Sph^2\) preserving
\(\Sigma_S\); moreover \(D\neq0\) on \(-\Gamma\), since \(D(-v)=-D(v)\), and
\(-\Gamma\neq\Gamma\) by the second assertion. The components on which
\(D\neq0\) therefore occur in antipodal pairs.
\end{proof}

\begin{definition}\label{def:regularcomponent}
A \emph{regular periodic abnormal component} is a connected component
\(\Gamma\) of \(\Sigma_S\) such that every point of \(\Gamma\) is regular,
\(D\neq0\) on \(\Gamma\) and \(Y_s\) does not vanish on \(\Gamma\).
\end{definition}

By \Cref{cor:sphereinvariance}, a regular periodic abnormal component is a
compact one-dimensional manifold invariant under \(Y_s\) and free of
equilibria, hence a periodic orbit of \(Y_s\), and of \(\widehat X_s\) by
\eqref{eq:YDX}. By \Cref{prop:antipodal}, \(-\Gamma\) is another regular
periodic abnormal component, disjoint from \(\Gamma\). We denote by
\(T_\Gamma^\sigma\) and \(T_\Gamma^\tau=\int_0^{T^\sigma_\Gamma}\abs{D}\dd\sigma\)
its periods for \(Y_s\) and \(\widehat X_s\).

\begin{proposition}[Persistence]\label{prop:persistence}
Let \(\Gamma\) be a regular periodic abnormal component for \((H,b)\), let
\(\Psi:\R/\mathbb Z\times(-r_0,r_0)\to\Sph^2\) be a tubular parametrization of a
neighbourhood of \(\Gamma\) with \(\Psi(s,0)\in\Gamma\), and set
\(\cZ_r:=\Psi\bigl(\R/\mathbb Z\times(-r,r)\bigr)\). There is \(r_1\in(0,r_0)\)
such that, for every \(r\in(0,r_1]\), there is \(\varepsilon>0\) with the
following property: for every \((\wt H,\wt b)\) whose coefficients are
\(\varepsilon\)-close to those of \((H,b)\), the corresponding spherical quintic
\(\wt\Sigma_S\) has a unique connected component \(\wt\Gamma\) meeting \(\cZ_r\).
Moreover \(\wt\Gamma\subset\cZ_r\), \(\wt\Gamma\) is a regular periodic abnormal
component for \((\wt H,\wt b)\), and \(\wt\Gamma=\Psi\bigl(\set{r=\wt g(s)}\bigr)\)
for a \(C^1\) function \(\wt g\) such that \(\norm{\wt g}_{C^1}\to0\) as
\(\varepsilon\to0\).
\end{proposition}

\begin{proof}
Since every point of \(\Gamma\) is regular, \(\Gamma\) is a compact embedded
smooth curve in \(\Sph^2\), and a tubular parametrization \(\Psi\) as in the
statement exists; we may take \(\Psi(s,r)=\exp_{\gamma(s)}\bigl(r\,\nu(s)\bigr)\),
where \(\gamma\) is a parametrization of \(\Gamma\) by \(\R/\mathbb Z\), \(\nu(s)\)
is a unit normal to \(\Gamma\) in \(T_{\gamma(s)}\Sph^2\), and \(r_0\) is small
enough for \(\Psi\) to be a smooth diffeomorphism onto its image. Let
\(F:=D''\circ\Psi\). Since \(\Gamma\subset\Sigma_S\), we have \(F(s,0)=0\) for
all \(s\), hence \(\partial_sF(s,0)=0\). Since \(d\Psi\) is invertible and the
differential of \(D''|_{\Sph^2}\) does not vanish on \(\Gamma\), while its
restriction to the tangent direction of \(\Gamma\) vanishes, we also have
\(\partial_rF(s,0)\neq0\) for every \(s\). By continuity and compactness of
\(\R/\mathbb Z\), and since the set \(\R/\mathbb Z\times[-r_1,r_1]\) is connected,
we can choose \(r_1\in(0,r_0)\) and \(m>0\) such that, after possibly
replacing \(\nu\) by \(-\nu\), \(\partial_rF\ge2m\) on
\(\R/\mathbb Z\times[-r_1,r_1]\). Since \(D\) and \(Y_s\) do not vanish on the
compact set \(\Gamma\), we may also assume, reducing \(r_1\) if necessary, that
\(\abs{D}\ge2a\) and \(\abs{Y_s}\ge2a\) on \(\ov{\cZ_{r_1}}\) for some \(a>0\).
In particular \(\Sigma_S\cap\ov{\cZ_{r_1}}=\Gamma\), since for each \(s\) the
function \(r\mapsto F(s,r)\) is strictly increasing on \([-r_1,r_1]\) and
vanishes at \(r=0\).

The polynomials \(D\), \(D'\), \(D''\) and the components of \(Z_s\) are
polynomial expressions in the coefficients of \((H,b)\), and they belong to
finite-dimensional spaces of homogeneous polynomials on which all norms are
equivalent. Consequently \(\wt D\to D\), \(\wt D''\to D''\) and
\(\wt Z_s\to Z_s\) in \(C^1(\Sph^2)\) as the coefficients of \((\wt H,\wt b)\)
converge to those of \((H,b)\); the same holds for
\(\wt Y_s=\wt Z_s-\pair{\wt Z_s}{v}v\), and \(\wt F:=\wt D''\circ\Psi\) converges
to \(F\) in \(C^1\bigl(\R/\mathbb Z\times[-r_1,r_1]\bigr)\). Fix
\(r\in(0,r_1]\) and \(\delta\in(0,mr/2)\). There is \(\varepsilon>0\) such that,
whenever the coefficients are \(\varepsilon\)-close, we have
\(\norm{\wt F-F}_{C^1}<\delta\), \(\delta<m\), \(\abs{\wt D}\ge a\) and
\(\abs{\wt Y_s}\ge a\) on \(\ov{\cZ_{r_1}}\). Then
\(\partial_r\wt F\ge2m-\delta\ge m\) on \(\R/\mathbb Z\times[-r_1,r_1]\), so
that for every \(s\) the function \(r'\mapsto\wt F(s,r')\) is strictly
increasing on \([-r_1,r_1]\). Since \(\abs{\wt F(s,0)}=\abs{\wt F(s,0)-F(s,0)}<\delta\),
the mean value theorem gives \(\wt F(s,\delta/m)\ge\wt F(s,0)+\delta>0\) and
\(\wt F(s,-\delta/m)\le\wt F(s,0)-\delta<0\). Hence for every \(s\) there is
exactly one \(\wt g(s)\in[-r_1,r_1]\) with \(\wt F(s,\wt g(s))=0\), and it
satisfies \(\abs{\wt g(s)}<\delta/m<r/2\). As \(\partial_r\wt F\neq0\), the
implicit function theorem shows that \(\wt g\) is of class \(C^1\), with
\[
\wt g'(s)=-\frac{\partial_s\wt F\bigl(s,\wt g(s)\bigr)}{\partial_r\wt F\bigl(s,\wt g(s)\bigr)} .
\]
Since \(\partial_sF(s,0)=0\), we can estimate
\(\abs{\partial_s\wt F(s,\wt g(s))}\le\norm{\wt F-F}_{C^1}
+\abs{\partial_sF(s,\wt g(s))-\partial_sF(s,0)}\), and the right-hand side
tends to zero uniformly in \(s\) as \(\varepsilon\to0\), by the uniform
continuity of \(\partial_sF\) and the bound \(\abs{\wt g}<\delta/m\), where
\(\delta\) may be taken arbitrarily small. Together with
\(\partial_r\wt F\ge m\), this shows that \(\norm{\wt g}_{C^1}\to0\) as
\(\varepsilon\to0\).

We have proved that \(\wt\Sigma_S\cap\ov{\cZ_{r_1}}\) is the embedded circle
\(\wt\Gamma:=\Psi\bigl(\set{r'=\wt g(s)}\bigr)\), which is contained in
\(\cZ_{r/2}\). In particular \(\wt\Gamma=\wt\Sigma_S\cap\cZ_{r_1}\) is open in
\(\wt\Sigma_S\), and it is compact, hence closed; being connected, it is a
connected component of \(\wt\Sigma_S\). Any connected component of
\(\wt\Sigma_S\) meeting \(\cZ_r\) contains a point of \(\wt\Gamma\), and
therefore coincides with \(\wt\Gamma\); this proves uniqueness. Every point of
\(\wt\Gamma\) is regular, because \(\partial_r\wt F\neq0\) there and \(\Psi\) is
a diffeomorphism, so that the differential of \(\wt D''|_{\Sph^2}\) does not
vanish on \(\wt\Gamma\). Moreover \(\abs{\wt D}\ge a\) and \(\abs{\wt Y_s}\ge a\)
on \(\wt\Gamma\subset\ov{\cZ_{r_1}}\).

It remains to show that \(\wt\Gamma\) is a periodic orbit of \(\wt Y_s\).
\Cref{cor:sphereinvariance}, applied to \((\wt H,\wt b)\), shows that
\(\wt\Sigma_S\) is invariant under the flow \(\wt\Phi_\sigma\) of \(\wt Y_s\).
For \(w\in\wt\Gamma\), the set of \(\sigma\in\R\) such that
\(\wt\Phi_\sigma(w)\in\wt\Gamma\) is closed, by compactness of \(\wt\Gamma\), and
open, because \(\wt\Gamma\) is open in \(\wt\Sigma_S\) and
\(\sigma\mapsto\wt\Phi_\sigma(w)\) is a continuous curve in \(\wt\Sigma_S\); as it
contains \(\sigma=0\), it is all of \(\R\), and \(\wt\Gamma\) is invariant. The
restriction of \(\wt Y_s\) to \(\wt\Gamma\) is then a nowhere vanishing vector
field on an embedded circle. The orbit of any point of \(\wt\Gamma\) under the
restricted flow is open in \(\wt\Gamma\), since the flow is a local
diffeomorphism transverse to no direction other than that of the circle,
and its complement is a union of orbits, hence also open; by connectedness
the orbit is the whole circle. A flow on a circle with a single orbit is
periodic, so \(\wt\Gamma\) is a periodic orbit of \(\wt Y_s\), and therefore a
regular periodic abnormal component for \((\wt H,\wt b)\).
\end{proof}

Let \(\Gamma\) be a regular periodic abnormal component and let \(\cZ\) be a
tubular neighbourhood of \(\Gamma\) on which \(D\neq0\) and \(\Gamma\) is the only
zero of \(D''\), as in the proof of \Cref{prop:persistence}. Then
\(\cZ\setminus\Gamma=\cZ_+\cup\cZ_-\), where \(D''\) has constant opposite signs
on the two sides. Through every point of \(\cZ\setminus\Gamma\) passes a unique
singular trajectory, which is normal by \Cref{lem:singular} since
\(D''\neq0\); it is hyperbolic on the side where \(DD''>0\) and elliptic on
the other side. By \Cref{cor:sphereinvariance}, both sides are invariant under
the singular flow: a normal singular arc never crosses the abnormal cone, and
both types of arcs are always present near \(\Gamma\).

\begin{theorem}[Characteristic multiplier]\label{thm:multiplier}
Let \(\gamma\) be a parametrization of \(\Gamma\) by \(Y_s\), with minimal period
\(T^\sigma_\Gamma\), and set
\[
\lambda_\Gamma:=\int_0^{T^\sigma_\Gamma}k(\gamma(\sigma))\dd\sigma,
\qquad k=\cK-5\kappa .
\]
Then \(\lambda_\Gamma\) does not depend on the parametrization, and the Poincar{\'e}
return map of \(Y_s\) on a transversal section to \(\Gamma\) has derivative
\(e^{\lambda_\Gamma}\) at \(\Gamma\). Consequently:
\begin{enumerate}[label=\textnormal{(\roman*)}]
\item if \(\lambda_\Gamma<0\) (resp.\ \(>0\)), \(\Gamma\) is a hyperbolic
attracting (resp.\ repelling) limit cycle of \(Y_s\), and all the normal
singular trajectories starting in a neighbourhood of \(\Gamma\), hyperbolic and
elliptic, accumulate on \(\Gamma\) as \(\sigma\to+\infty\) (resp.\ \(-\infty\));
\item with respect to the spherical time \(\tau\) the multiplier is
\(e^{\operatorname{sgn}(D|_\Gamma)\lambda_\Gamma}\), and
\(\lambda_{-\Gamma}=-\lambda_\Gamma\) with respect to \(\sigma\).
\end{enumerate}
\end{theorem}

\begin{proof}
Two parametrizations of \(\Gamma\) by \(Y_s\) differ by a translation of time,
and the integral of the periodic function \(k\circ\gamma\) over a period does not
depend on its starting point; hence \(\lambda_\Gamma\) is well defined. Denote by
\(\Phi_\sigma\) the flow of \(Y_s\). Since \(Y_s(\gamma(0))\neq0\), there is a smooth
arc \(\Pi\) through \(\gamma(0)\) transverse to \(Y_s\). On \(\Sigma_S\) we have
\(dD''(Y_s)=kD''=0\) by \Cref{cor:sphereinvariance}, so the differential of
\(D''|_{\Sph^2}\) at \(\gamma(0)\), which is nonzero by regularity, vanishes on
\(Y_s(\gamma(0))\) and is therefore nonzero on the tangent direction of \(\Pi\).
Shrinking \(\Pi\), the restriction \(z:=D''|_\Pi\) is a smooth coordinate on \(\Pi\)
with \(z(\gamma(0))=0\). By the standard construction of the Poincar{\'e} map,
based on the implicit function theorem, there are a neighbourhood
\(\Pi_0\subset\Pi\) of \(\gamma(0)\) and a \(C^1\) function \(T:\Pi_0\to(0,\infty)\)
with \(T(\gamma(0))=T^\sigma_\Gamma\) such that \(P(x):=\Phi_{T(x)}(x)\in\Pi\) for
\(x\in\Pi_0\), and \(P\) is of class \(C^1\), with \(P(\gamma(0))=\gamma(0)\). The
derivative of \(P\) at its fixed point does not depend on the coordinate on
\(\Pi\), since a change of coordinates conjugates \(P\) by a diffeomorphism
fixing \(\gamma(0)\). By \Cref{cor:sphereinvariance},
\[
z(P(x))=D''\bigl(\Phi_{T(x)}(x)\bigr)
=z(x)\,E(x),\qquad
E(x):=\exp\int_0^{T(x)}k(\Phi_s x)\dd s>0 .
\]
The function \(E\) is continuous on \(\Pi_0\), by continuity of \(T\), of the flow
and of \(k\), and \(E(\gamma(0))=e^{\lambda_\Gamma}\). Writing the return map in the
coordinate \(z\) as \(p(z)=z\,E(x(z))\), we obtain \(p(0)=0\) and
\(p'(0)=\lim_{z\to0}p(z)/z=e^{\lambda_\Gamma}\).

Assume \(\lambda_\Gamma<0\), and choose \(q\in(e^{\lambda_\Gamma},1)\). By continuity
of \(E\) there is \(\delta>0\) such that \(E\le q\) for \(\abs{z}\le\delta\), so that
\(\abs{p(z)}\le q\abs{z}\) on this interval; moreover \(p\) preserves the sign of
\(z\), because \(E>0\). Let \(x\in\Pi_0\) with \(\abs{z(x)}\le\delta\), and set
\(x_n:=P^n(x)\) and \(\sigma_n:=\sum_{j<n}T(x_j)\). Then
\(\abs{z(x_n)}\le q^n\abs{z(x)}\), so that \(x_n\to\gamma(0)\), \(T(x_n)\to T^\sigma_\Gamma\)
and \(\sigma_n\to\infty\), while all the \(x_n\) lie on the same side of \(\Gamma\)
as \(x\). By uniform continuity of \((s,y)\mapsto\Phi_s(y)\) on
\([0,T^\sigma_\Gamma+1]\times\Sph^2\), the distance between \(\Phi_s(x_n)\) and
\(\gamma(s)\) tends to zero uniformly in \(s\in[0,T^\sigma_\Gamma+1]\). Since every
\(\sigma\ge\sigma_n\) can be written as \(\sigma=\sigma_m+s\) with \(m\ge n\) and
\(s\in[0,T(x_m)]\), and \(\Phi_\sigma(x)=\Phi_s(x_m)\), the distance from
\(\Phi_\sigma(x)\) to \(\Gamma\) tends to zero as \(\sigma\to\infty\); and every point
\(\gamma(s)\) is the limit of \(\Phi_{\sigma_n+s}(x)=\Phi_s(x_n)\). Thus the
\(\omega\)-limit set of \(x\) is \(\Gamma\). It remains to show that forward orbits
of all points close to \(\Gamma\) reach such an \(x\). By the flow-box theorem
there are \(\eta>0\) and a neighbourhood \(W\) of \(\gamma(0)\) such that every
\(y\in W\) satisfies \(\Phi_s(y)\in\Pi_0\) for some \(s\in(-\eta,\eta)\), with
\(\abs{z(\Phi_s(y))}\le\delta\) if \(W\) is small; we may assume
\(\eta<T^\sigma_\Gamma\). If \(y\) is close enough to \(\Gamma\), it is close to
some \(\gamma(s_0)\) with \(s_0\in[0,T^\sigma_\Gamma)\), and by uniform continuity
of the flow on \([0,2T^\sigma_\Gamma]\times\Sph^2\) the point
\(\Phi_{2T^\sigma_\Gamma-s_0}(y)\) lies in \(W\); its orbit then meets \(\Pi_0\) at
the positive time \(2T^\sigma_\Gamma-s_0+s>0\), at a point \(x\) as above. Hence
every forward orbit starting in a neighbourhood of \(\Gamma\) accumulates on
\(\Gamma\) from the side on which it starts, which proves (i) for hyperbolic and
for elliptic normal singular trajectories. If \(\lambda_\Gamma>0\), the same
argument applied to \(-Y_s\), whose return map is \(P^{-1}\), with multiplier
\(e^{-\lambda_\Gamma}<1\), gives the statement as \(\sigma\to-\infty\).

On \(\cZ\) we have \(\widehat X_s=Y_s/D\) by \eqref{eq:YDX}, and \(D\) has the
constant sign \(\epsilon=\operatorname{sgn}(D|_\Gamma)\). The field
\(\widehat X_s\) is thus a positive time change of \(\epsilon Y_s\): it has the same
oriented orbits as \(\epsilon Y_s\), so its first return map to \(\Pi\) is \(P\) if
\(\epsilon=1\) and \(P^{-1}\) if \(\epsilon=-1\), with multiplier
\(e^{\epsilon\lambda_\Gamma}\). Finally, \(\cK\) and \(\kappa\) are homogeneous of
degrees five and seven, so \(k\) is odd, and by \Cref{prop:antipodal}(i) the
curve \(\tilde\gamma(\sigma)=-\gamma(-\sigma)\) parametrizes \(-\Gamma\) by \(Y_s\),
with the same minimal period. Hence
\[
\lambda_{-\Gamma}=\int_0^{T^\sigma_\Gamma}k\bigl(-\gamma(-\sigma)\bigr)\dd\sigma
=-\int_{-T^\sigma_\Gamma}^0k(\gamma(s))\dd s=-\lambda_\Gamma ,
\]
by periodicity of \(\gamma\).
\end{proof}

\begin{remark}
For a periodic orbit of a \(C^1\) vector field \(Y\) on a surface endowed with
an area form \(\omega\), the multiplier equals \(\exp\int_\Gamma\dvg_\omega Y\dd\sigma\);
in the plane with the Lebesgue measure this is the classical formula, see
\cite{DLA06}, and the general case follows because replacing \(\omega\) by
\(h\omega\), \(h>0\), changes the divergence by \(Y(\log h)\), whose integral along
a closed orbit vanishes. With the area form of \(\Sph^2\),
\Cref{thm:multiplier} therefore yields
\[
\int_\Gamma k\dd\sigma=\int_\Gamma\dvg_{\Sph^2}Y_s\dd\sigma ;
\]
the left-hand side only involves the algebraic data \(D,D',D''\) and the
divergence of \(H\) along \(\Gamma\). The numerical example of
\Cref{sec:numerics} confirms this identity to ten digits. Note also that
\(\lambda_\Gamma/T^\sigma_\Gamma=\int k\dd\mu_\Gamma\) is the average of \(k\) with
respect to the invariant measure of \(\Gamma\); this observation is the
starting point of \Cref{prop:transverse}, where the transverse stability of
general abnormal invariant sets is expressed as a singular Mather problem.
\end{remark}

\begin{corollary}[Robustness]\label{cor:robust}
Let \(\Gamma\) be a regular periodic abnormal component for \((H,b)\) with
\(\lambda_\Gamma\neq0\). Then, for \((\wt H,\wt b)\) close enough to \((H,b)\), the
component \(\wt\Gamma\) of \Cref{prop:persistence} satisfies
\(\operatorname{sgn}\lambda_{\wt\Gamma}=\operatorname{sgn}\lambda_\Gamma\); in
particular it is a hyperbolic limit cycle of the same type.
\end{corollary}

\begin{proof}
Parametrize \(\Gamma\) by arc length \(s\in[0,L]\). Since \(d\sigma=ds/\abs{Y_s}\)
along the orbit, \(\lambda_\Gamma=\int_0^Lk/\abs{Y_s}\dd s\), and similarly for
\(\wt\Gamma\). By \Cref{prop:persistence}, \(\wt\Gamma\) is the image under the
tubular parametrization of the graph of a function \(\wt g\) with
\(\norm{\wt g}_{C^1}\to0\), so that its arc-length parametrization and its
length converge to those of \(\Gamma\) in \(C^0\), while \(\wt k\to k\) and
\(\wt Y_s\to Y_s\) uniformly on \(\Sph^2\), with \(\abs{Y_s}\) bounded below near
\(\Gamma\). Hence \(\lambda_{\wt\Gamma}\to\lambda_\Gamma\), and the sign is preserved.
\end{proof}

\subsection{Radial dynamics, physical time and admissibility}

The spherical fields only describe the direction of the singular
trajectories. The following proposition reconstructs the
physical time and the radius, and will provide the physical meaning of the
weak KAM objects of \Cref{sec:periodic}.

\begin{proposition}\label{prop:lift}
Let \(w:I\to\Sph^2\cap\set{D\neq0}\) be a trajectory of \(\widehat X_s\), defined
on an interval \(I\ni0\), and let \(r_0>0\). Define \(r(\tau)\) and \(t(\tau)\),
for \(\tau\in I\), by
\begin{equation}\label{eq:radial}
\log r(\tau)=\log r_0+\int_0^\tau\rho(w(s))\dd s,\qquad
t(\tau)=\int_0^\tau r(s)^{-2}\dd s .
\end{equation}
Then \(t\) is an increasing diffeomorphism of \(I\) onto an interval \(t(I)\ni0\),
with inverse \(\tau(t)\), and \(v(t):=r(\tau(t))\,w(\tau(t))\), \(t\in t(I)\), is
the solution of \(\dot v=X_s(v)\) with \(v(0)=r_0w(0)\); moreover
\(\abs{u_s(v(t))}=r^3\abs{D'/D}(w)\). If \(w\) runs along a regular periodic
abnormal component \(\Gamma\), let
\[
\ov\rho_\Gamma:=\frac1{T_\Gamma^\tau}\int_0^{T_\Gamma^\tau}\rho(w(s))\dd s,
\qquad
\vartheta(\tau):=\int_0^\tau\bigl(\rho(w(s))-\ov\rho_\Gamma\bigr)\dd s .
\]
Then \(\vartheta\) is \(T^\tau_\Gamma\)-periodic,
\(\log r(\tau)=\log r_0+\ov\rho_\Gamma\tau+\vartheta(\tau)\), and:
\begin{enumerate}[label=\textnormal{(\alph*)}]
\item if \(\ov\rho_\Gamma>0\), the abnormal trajectory blows up at the finite
time \(t_*=r_0^{-2}\int_0^\infty e^{-2\ov\rho_\Gamma s-2\vartheta(s)}\dd s\);
\item if \(\ov\rho_\Gamma<0\), it is defined for all \(t\ge0\) and converges to
the origin, the spherical clock satisfying \(\tau(t)\to\infty\);
\item if \(\ov\rho_\Gamma=0\), it is a periodic orbit of \(X_s\) in \(\R^3\), of
period \(t(T^\tau_\Gamma)=r_0^{-2}\int_0^{T^\tau_\Gamma}e^{-2\vartheta(s)}\dd s\);
\item if \(\ov\rho_\Gamma\le0\) and
\(r_0\le m_\Gamma^{-1/3}e^{-\osc\vartheta}\), where
\(m_\Gamma=\max_\Gamma\abs{D'/D}\) (with \(m_\Gamma^{-1/3}=+\infty\) if
\(m_\Gamma=0\)) and \(\osc\vartheta=\max\vartheta-\min\vartheta\), the abnormal
trajectory satisfies \(\abs{u_s}\le1\) for all \(t\ge0\).
\end{enumerate}
\end{proposition}

\begin{proof}
On the open cone \(\set{D\neq0}\) the field \(X_s=H-(D'/D)\,b\) is smooth and
positively homogeneous of degree three, since \(H\) is homogeneous of degree
three and \(D'/D\) is the quotient of homogeneous polynomials of degrees six
and three; thus \(X_s(\lambda v)=\lambda^3X_s(v)\) for \(\lambda>0\), and in the
same way \(u_s(\lambda v)=\lambda^3u_s(v)\). The functions \(r\) and \(t\) in
\eqref{eq:radial} are of class \(C^1\) on \(I\), with \(r>0\) and
\(t'(\tau)=r(\tau)^{-2}>0\), so that \(t\) is an increasing \(C^1\)
diffeomorphism of \(I\) onto the interval \(t(I)\), which contains \(0=t(0)\),
and its inverse \(\tau(t)\) satisfies \(\tau'(t)=r(\tau(t))^2\). Differentiating
\(\tilde v(\tau):=r(\tau)w(\tau)\) with respect to \(\tau\), and using
\(r'=\rho(w)\,r\) and \(w'=\widehat X_s(w)=X_s(w)-\rho(w)\,w\), we obtain
\[
\tilde v'(\tau)=r'w+rw'=r\rho(w)\,w+r\bigl(X_s(w)-\rho(w)\,w\bigr)=r\,X_s(w).
\]
Therefore, for \(t\in t(I)\),
\[
\dot v(t)=\tau'(t)\,\tilde v'(\tau(t))=r^2\cdot r\,X_s(w)=r^3X_s(w)=X_s(rw)=X_s(v(t)),
\]
where \(r\) and \(w\) are evaluated at \(\tau(t)\), and
\(v(0)=r(0)w(0)=r_0w(0)\). Since \(X_s\) is locally Lipschitz on
\(\set{D\neq0}\), \(v\) is the unique solution of \(\dot v=X_s(v)\) with this
initial condition on \(t(I)\). The identity
\(\abs{u_s(v(t))}=r^3\abs{u_s(w)}=r^3\abs{D'/D}(w)\) follows from the
homogeneity of \(u_s\).

Assume now that \(w\) runs along a regular periodic abnormal component
\(\Gamma\). Then \(D\neq0\) on \(\Gamma\), \(\Gamma\) is a periodic orbit of
\(\widehat X_s\) by \eqref{eq:YDX}, and \(w\) is defined on \(I=\R\) and is
\(T^\tau_\Gamma\)-periodic. The function \(\rho\circ w\) is continuous and
\(T^\tau_\Gamma\)-periodic, so that
\[
\vartheta(\tau+T^\tau_\Gamma)-\vartheta(\tau)
=\int_\tau^{\tau+T^\tau_\Gamma}\rho(w(s))\dd s-\ov\rho_\Gamma T^\tau_\Gamma
=\int_0^{T^\tau_\Gamma}\rho(w(s))\dd s-\ov\rho_\Gamma T^\tau_\Gamma=0,
\]
by definition of \(\ov\rho_\Gamma\); hence \(\vartheta\) is periodic, in
particular bounded, and since \(\vartheta(0)=0\) we have
\(\min\vartheta\le0\le\max\vartheta\). The identity
\(\log r(\tau)=\log r_0+\ov\rho_\Gamma\tau+\vartheta(\tau)\) is a rewriting of
\eqref{eq:radial}, and it gives
\[
t(\tau)=r_0^{-2}\int_0^\tau e^{-2\ov\rho_\Gamma s-2\vartheta(s)}\dd s,
\qquad
e^{-2\max\vartheta}e^{-2\ov\rho_\Gamma s}
\le e^{-2\ov\rho_\Gamma s-2\vartheta(s)}
\le e^{-2\min\vartheta}e^{-2\ov\rho_\Gamma s}.
\]

If \(\ov\rho_\Gamma>0\), the upper bound is integrable on \([0,\infty)\), so
\(t(\tau)\) increases to the finite limit \(t_*\) of the statement as
\(\tau\to\infty\), and \(t(\R_+)=[0,t_*)\). Moreover
\(r(\tau(t))\ge r_0e^{\min\vartheta}e^{\ov\rho_\Gamma\tau(t)}\to\infty\) as
\(t\to t_*\), because \(\tau(t)\to\infty\). Thus \(\abs{v(t)}\to\infty\) as
\(t\to t_*\), the solution cannot be extended beyond \(t_*\), and it blows up
at \(t_*\).

If \(\ov\rho_\Gamma<0\), the lower bound shows that
\(t(\tau)\ge r_0^{-2}e^{-2\max\vartheta}\int_0^\tau e^{2\abs{\ov\rho_\Gamma}s}\dd s\to\infty\)
as \(\tau\to\infty\), so that \(t(\R_+)=[0,\infty)\), the solution is defined
for all \(t\ge0\), and \(\tau(t)\to\infty\) as \(t\to\infty\). Since
\(\abs{v(t)}=r(\tau(t))\le r_0e^{\max\vartheta}e^{\ov\rho_\Gamma\tau(t)}\), we
conclude that \(v(t)\to0\).

If \(\ov\rho_\Gamma=0\), then \(r(\tau)=r_0e^{\vartheta(\tau)}\) is
\(T^\tau_\Gamma\)-periodic, and so is the integrand in the expression of
\(t(\tau)\). Setting \(T_t:=t(T^\tau_\Gamma)>0\), we obtain
\(t(\tau+T^\tau_\Gamma)=t(\tau)+T_t\) for every \(\tau\in\R\); in particular
\(t(\R)=\R\), and, \(t\) being a bijection, \(\tau(t+T_t)=\tau(t)+T^\tau_\Gamma\).
Consequently
\(v(t+T_t)=r(\tau(t)+T^\tau_\Gamma)\,w(\tau(t)+T^\tau_\Gamma)=v(t)\) for all
\(t\in\R\). Since \(w\) is not constant, \(v\) is a nonconstant periodic
solution of \(\dot v=X_s(v)\), of period \(T_t\), whose orbit lies on the cone
over \(\Gamma\).

Finally, assume \(\ov\rho_\Gamma\le0\). By the two previous cases the solution
is defined for all \(t\ge0\), and for \(\tau\ge0\)
\[
\log r(\tau)-\log r_0=\ov\rho_\Gamma\tau+\vartheta(\tau)\le\max\vartheta
\le\max\vartheta-\min\vartheta=\osc\vartheta,
\]
where we used \(\min\vartheta\le0\). Hence
\(r(\tau)\le r_0e^{\osc\vartheta}\le m_\Gamma^{-1/3}\), and
\(\abs{u_s(v(t))}=r^3\abs{D'/D}(w)\le r^3m_\Gamma\le1\) for every \(t\ge0\).
\end{proof}

\begin{remark}[Clocks]\label{rem:clocks}
The density \(\abs{D}=\abs{d\tau/d\sigma}\) is the ratio between the
polynomial clock \(\sigma\) and the spherical clock \(\tau\), and
the physical time is \(dt=r^{-2}d\tau\). Costs which are meant to measure a
physical quantity should therefore be expressed with respect to \(\tau\) (or
\(t\)); \Cref{prop:covariance} below explains how the Aubry--Mather
structure transforms when the clock is changed.
\end{remark}

% =================================================================
\section{Singular Aubry--Mather structures}\label{sec:AMstructures}
% =================================================================

In this section the discussion is independent of the cubic model. Let \(M\)
be a compact smooth manifold, \(U\subset M\) an open set, \(Y\) a \(C^1\)
vector field on \(U\), and \(K\subset U\) a nonempty compact set invariant
under the flow \(\Phi_t\) of \(Y\) (so that \(\Phi_t\) is defined on \(K\) for
all \(t\in\R\)). Let \(\ell\in C(K)\) be a running cost. We denote by
\(\cM_Y(K)\) the set of \(\Phi\)-invariant Borel probability measures on \(K\),
which is nonempty, convex and weak-\(*\) compact.

\begin{example}\label{ex:settings}
The two situations of interest for \eqref{eq:system} are:
\begin{enumerate}[label=\textnormal{(\alph*)}]
\item \(M=U=\Sph^2\), \(Y=Y_s\), and \(K\subset\Sigma_S\) compact and invariant
(for instance \(K=\Sigma_S\), or a union of periodic components and
equilibria);
\item \(M=\Sph^2\), \(U=\Sph^2\cap\set{D\neq0}\), \(Y=\widehat X_s\), and \(K\) a
compact invariant subset of the \emph{admissible abnormal set}
\(\Sigma_S\cap\set{D\neq0}\).
\end{enumerate}
Natural costs are the clock density \(\ell=\abs{D}\) in case (a), and the
radial exponent \(\ell=\rho\) of \Cref{prop:lift} in case (b). More
generally, for a single-input system in \(\R^n\) the same framework applies
to the singular feedback flow restricted to any compact invariant subset of
the singular locus.
\end{example}

\subsection{Singular Lagrangian, critical value and Aubry--Mather sets}

\begin{definition}\label{def:singLag}
The \emph{singular Lagrangian} of \((K,Y,\ell)\) is
\(L:TM\to\R\cup\set{+\infty}\),
\[
L(v,\xi)=\ell(v)\ \text{ if } v\in K \text{ and } \xi=Y(v),\qquad
L(v,\xi)=+\infty\ \text{ otherwise.}
\]
\end{definition}

\(L\) is the Lagrangian of the differential inclusion \(\dot v=Y(v)\) with
running cost \(\ell\); its Fenchel transform in the fibres is
\(H_0(v,p)=\pair{p}{Y(v)}-\ell(v)\), \(v\in K\), which is linear in \(p\). In
particular \(L\) is not of Tonelli type, and the variational problem has no
freedom in the velocity: all the selection happens at the level of
measures. Following Ma{\~n}{\'e} and Bernard \cite{Mane97,Bernard08}, a compactly
supported probability measure \(\nu\) on \(TM\) is \emph{closed} if
\(\int d\varphi_v(\xi)\dd\nu(v,\xi)=0\) for every \(\varphi\in C^1(M)\).

\begin{lemma}\label{lem:infinv}
A probability measure \(\mu\) on \(K\) is invariant if and only if
\(\int_K d\varphi(Y)\dd\mu=0\) for every \(\varphi\in C^1(M)\).
\end{lemma}

\begin{proof}
Fix \(t\in\R\). Since \(K\) is compact and invariant, \(\Phi_t\) is defined and
of class \(C^1\) on an open neighbourhood \(V_t\) of \(K\). Let
\(\theta\in C^\infty_c(V_t)\) be equal to one near \(K\) and
\(\psi:=\theta\cdot(\varphi\circ\Phi_t)\in C^1(M)\). Since \(\Phi_t\) maps
\(Y\) to itself, on \(K\) we have
\(d\psi(Y)=d(\varphi\circ\Phi_t)(Y)=d\varphi(Y)\circ\Phi_t\). If the condition
holds, \(F(t):=\int\varphi\circ\Phi_t\dd\mu\) satisfies
\(F'(t)=\int d\varphi(Y)\circ\Phi_t\dd\mu=\int d\psi(Y)\dd\mu=0\), hence
\(\int\varphi\circ\Phi_t\dd\mu=\int\varphi\dd\mu\) for all \(\varphi\in C^1(M)\),
a dense subset of \(C(K)\). Conversely, if \(\mu\) is invariant then \(F\) is
constant and \(F'(0)=\int d\varphi(Y)\dd\mu=0\).
\end{proof}

\begin{proposition}\label{prop:closed}
A compactly supported probability measure \(\nu\) on \(TM\) is closed and has
finite action \(\int L\dd\nu<\infty\) if and only if
\(\nu=(\mathrm{id},Y)_\#\mu\) for some \(\mu\in\cM_Y(K)\). In that case
\(\int L\dd\nu=\int_K\ell\dd\mu\).
\end{proposition}

\begin{proof}
Finite action forces \(\nu\) to be concentrated on the graph
\(\set{(v,Y(v)):\ v\in K}\), hence \(\nu=(\mathrm{id},Y)_\#\mu\) with
\(\mu=\pi_\#\nu\). Closedness then reads \(\int d\varphi(Y)\dd\mu=0\) for all
\(\varphi\), and \Cref{lem:infinv} applies.
\end{proof}

\begin{definition}[Singular Aubry--Mather structure]\label{def:AMstructure}
\begin{enumerate}[label=\textnormal{(\roman*)}]
\item The \emph{critical value} of \((K,Y,\ell)\) is
\[
c=c(K,Y,\ell):=\min_{\mu\in\cM_Y(K)}\int_K\ell\dd\mu .
\]
A \emph{singular Mather measure} is a minimizer; the \emph{singular Mather
set} \(\wt\cM_s\) is the closure of the union of their supports.
\item The \emph{critical Hamiltonian} is \(\cH_s(v,p)=\pair{p}{Y(v)}-\ell(v)+c\),
\(v\in K\). A \emph{critical subsolution} is a function \(\chi\in C^1(M)\)
such that \(\cH_s(v,d\chi(v))\le0\) on \(K\), that is,
\[
g_\chi:=\ell-c-d\chi(Y)\ge0\qquad\text{on }K .
\]
We denote by \(\cS\) the convex set of critical subsolutions, and we call
\[
\cZ(\chi):=\set{v\in K:\ g_\chi(v)=0}
\]
the \emph{contact set} of \(\chi\in\cS\).
\item An orbit \(\Phi_\R(v)\) is \emph{\(\chi\)-calibrated} if
\[
\chi(\Phi_bv)-\chi(\Phi_av)=\int_a^b\bigl(\ell(\Phi_sv)-c\bigr)\dd s
\qquad\text{for all }a<b .
\]
When \(\cS\neq\emptyset\), the \emph{singular Aubry set} is
\[
\wt\cA_s:=\set{v\in K:\ \Phi_\R(v)\text{ is }\chi\text{-calibrated for every }\chi\in\cS}.
\]
\end{enumerate}
The data \((c,\wt\cM_s,\cS,\wt\cA_s)\) form the \emph{singular Aubry--Mather
structure} of \((K,Y,\ell)\).
\end{definition}

\begin{remark}\label{rem:conventions}
With Fathi's conventions \cite{Fathi}, the critical value of \(H_0\) is
\(-c\), and \(\cH_s=H_0+c\). The Lax--Oleinik semigroup of the singular
Lagrangian is the transfer operator
\(T_tu(v)=u(\Phi_{-t}v)+\int_{-t}^0\ell(\Phi_sv)\dd s\); its fixed points up to
the constant \(c\) are exact solutions of the cohomological equation
\(d\chi(Y)=\ell-c\), which in general do not exist on \(K\) (see
\Cref{prop:heteroclinic}). For this reason the theory is built on
subsolutions, exactly as sub-actions in ergodic optimization
\cite{Jenkinson,Bousch00,CLT01}; when \(K\) is a single periodic orbit,
exact solutions exist and the classical picture is recovered
(\Cref{thm:periodic}).
\end{remark}

\begin{theorem}[Critical value]\label{thm:duality} The following properties hold. 
\begin{enumerate}[label=\textnormal{(\roman*)}]
\item \(\displaystyle c=\sup_{\chi\in C^1(M)}\ \min_{v\in K}\bigl(\ell-d\chi(Y)\bigr)(v)\).
In particular, for every \(\varepsilon>0\) there is \(\chi\in C^1(M)\) with
\(g_\chi\ge-\varepsilon\) on \(K\).
\item \(\displaystyle c=\lim_{T\to\infty}\frac1T\min_{v\in K}\int_0^T\ell(\Phi_sv)\dd s
=\sup_{T>0}\frac1T\min_{v\in K}\int_0^T\ell(\Phi_sv)\dd s\).
\end{enumerate}
\end{theorem}

\begin{proof}
(i) For \(\chi\in C^1(M)\) and \(\mu\in\cM_Y(K)\), \Cref{lem:infinv} gives
\(\min_K(\ell-d\chi(Y))\le\int(\ell-d\chi(Y))\dd\mu=\int\ell\dd\mu\), so the
supremum is at most \(c\). Suppose that it is smaller than \(c-\varepsilon\)
for some \(\varepsilon>0\). Then the affine subspace
\(\mathcal E:=\set{\ell-c+\varepsilon-d\chi(Y):\ \chi\in C^1(M)}\) of \(C(K)\)
does not meet the open convex cone \(\cP:=\set{f\in C(K):\ \min_Kf>0}\). By the
Hahn--Banach theorem there are a nonzero signed measure \(\mu\) and
\(\alpha\in\R\) with \(\int f\dd\mu>\alpha\ge\int e\dd\mu\) for \(f\in\cP\),
\(e\in\mathcal E\). Since \(\cP\) is a cone, \(\alpha\le0\) and \(\mu\ge0\); we
normalize \(\mu\) to a probability measure. Since \(\mathcal E\) contains a
whole affine line in the direction of every \(d\chi(Y)\), the bound
\(\int e\dd\mu\le\alpha\) forces \(\int d\chi(Y)\dd\mu=0\) for all \(\chi\),
so \(\mu\in\cM_Y(K)\) by \Cref{lem:infinv}, and
\(\int(\ell-c+\varepsilon)\dd\mu\le0\), contradicting the definition of
\(c\).

(ii) Let \(m(T):=\min_{v\in K}\int_0^T\ell(\Phi_sv)\dd s\). Splitting the
integral shows \(m(T+S)\ge m(T)+m(S)\), and \(m\) is Lipschitz; by Fekete's
lemma \(m(T)/T\) converges to \(\sup_Tm(T)/T\). Integrating
\(\int_0^T\ell\circ\Phi_s\dd s\) against a Mather measure gives \(m(T)\le Tc\).
Conversely, let \(v_T\) realize \(m(T)\) and
\(\mu_T:=\frac1T\int_0^T\delta_{\Phi_sv_T}\dd s\). Any weak-\(*\) limit point
\(\mu\) of \(\mu_T\) as \(T\to\infty\) is invariant, because
\(\abs{\int\varphi\circ\Phi_r\dd\mu_T-\int\varphi\dd\mu_T}\le2\abs{r}\norm{\varphi}_\infty/T\),
and \(\int\ell\dd\mu=\lim m(T)/T\). Hence \(c\le\lim m(T)/T\).
\end{proof}

The supremum in (i) need not be attained, that is, \(\cS\) may be empty
(\Cref{prop:heteroclinic}); this is the analogue of the failure of the
``Ma{\~n}{\'e} lemma'' in ergodic optimization \cite{Bousch00,CLT01}. When \(\cS\neq
\emptyset\), the structure has the same rigidity as in the Tonelli case
\cite{Fathi,FS04}.

\begin{theorem}[Structure]\label{thm:structure}
Assume \(\cS\neq\emptyset\).
\begin{enumerate}[label=\textnormal{(\roman*)}]
\item For every \(\chi\in\cS\) and every Mather measure \(\mu\),
\(\supp\mu\subset\cZ(\chi)\).
\item There exists \(\chi_*\in\cS\) such that
\(\cZ(\chi_*)=\bigcap_{\chi\in\cS}\cZ(\chi)=\wt\cA_s\). In particular
\(\chi_*\) is strict outside the Aubry set, and \(\wt\cA_s\) is compact,
invariant and nonempty.
\item \(\wt\cM_s\subset\wt\cA_s\), and \(\mu\in\cM_Y(K)\) is a Mather measure if
and only if \(\supp\mu\subset\wt\cA_s\).
\item For \(\chi\in\cS\), \(v\in\wt\cA_s\) and \(t\in\R\),
\(\chi(\Phi_tv)-\chi(v)=\int_0^t(\ell(\Phi_sv)-c)\dd s\); in particular the
increments of critical subsolutions along \(\wt\cA_s\) do not depend on the
subsolution.
\end{enumerate}
\end{theorem}

\begin{proof}
(i) By \Cref{lem:infinv}, \(\int g_\chi\dd\mu=\int(\ell-c)\dd\mu=0\) and
\(g_\chi\ge0\) is continuous.

(ii) \emph{Construction of \(\chi_*\).} Since \(C^1(M)\) is separable, there
is a sequence \((\chi_n)\) dense in \(\cS\) for the \(C^1\) topology. Choose
\(a_n>0\) with \(\sum a_n=1\) and \(\sum a_n\norm{\chi_n}_{C^1}<\infty\), and set
\(\chi_*:=\sum a_n\chi_n\in C^1(M)\). Then \(g_{\chi_*}=\sum a_ng_{\chi_n}\ge0\),
so \(\chi_*\in\cS\) and \(\cZ(\chi_*)=\bigcap_n\cZ(\chi_n)\). If
\(g_\chi(v)>0\) for some \(\chi\in\cS\), then \(g_{\chi_n}(v)>0\) for \(\chi_n\)
close to \(\chi\); hence \(\cZ(\chi_*)=\bigcap_{\chi\in\cS}\cZ(\chi)\).

\emph{Invariance of \(\cZ(\chi_*)\).} Let \(v\in\cZ(\chi_*)\) and assume by
contradiction that \(\Phi_{t_0}v\notin\cZ(\chi_*)\) for some \(t_0\), say
\(t_0>0\) (the case \(t_0<0\) is symmetric). Then \(Y(v)\neq0\), and replacing
\(t_0\) by its residue modulo the period if \(v\) is periodic, the orbit
segment \(\set{\Phi_sv:\ s\in[-\delta,t_0+\delta]}\) is an embedded arc for
\(\delta>0\) small. By the flow-box theorem there are a neighbourhood \(N\) of
this arc and \(C^1\) coordinates \((s,y)\in(-\delta,t_0+\delta)\times(-\eta,\eta)^{n-1}\)
on \(N\) in which \(Y=\partial_s\) and \(\Phi_sv=(s,0)\). Let
\(\theta_0,\theta_1\ge0\) be smooth bumps with equal integrals, supported near
\(s=0\) and near \(s=t_0\) respectively, with \(\theta_0(0)>0\), let
\(h(s)=\int_{-\delta}^s(\theta_1-\theta_0)\), which has compact support, and let
\(\beta\ge0\) be a bump in \(y\) with \(\beta(0)=1\). Choosing the supports small
we may assume \(g_{\chi_*}\ge\eta_0>0\) on
\(K\cap\supp(\theta_1\otimes\beta)\). For \(\varepsilon>0\) small,
\(\chi':=\chi_*+\varepsilon\,\beta(y)h(s)\) satisfies
\[
g_{\chi'}=g_{\chi_*}+\varepsilon\beta\theta_0-\varepsilon\beta\theta_1\ge0
\quad\text{on }K,\qquad g_{\chi'}(v)=\varepsilon\theta_0(0)>0,
\]
so \(\chi'\in\cS\) and \(v\notin\cZ(\chi')\), a contradiction.

\emph{Identification with \(\wt\cA_s\).} If \(v\in\wt\cA_s\) and \(\chi\in\cS\),
then \(\int_a^bg_\chi(\Phi_sv)\dd s=0\) for all \(a<b\), hence \(g_\chi=0\) along
the orbit and \(v\in\cZ(\chi)\). Conversely, if \(v\in\cZ(\chi_*)\), its orbit
is contained in \(\cZ(\chi_*)\subset\cZ(\chi)\) for every \(\chi\in\cS\), and
integrating \(g_\chi=0\) along the orbit shows that it is \(\chi\)-calibrated.
Compactness follows from the continuity of \(g_{\chi_*}\), and nonemptiness
from (i) and the existence of Mather measures.

(iii) The inclusion follows from (i) and (ii). If \(\mu\) is invariant and
\(\supp\mu\subset\wt\cA_s=\cZ(\chi_*)\), then
\(\int(\ell-c)\dd\mu=\int g_{\chi_*}\dd\mu=0\).

(iv) This is the calibration identity on \(\wt\cA_s\).
\end{proof}

\subsection{Change of clock}

Two observations are used repeatedly below. First, adding a constant
\(\alpha\) to the cost does not change the invariant measures, shifts the
critical value by \(\alpha\), and leaves the functions
\(g_\chi=\ell-c-d\chi(Y)\) unchanged; hence \((K,Y,\ell)\) and
\((K,Y,\ell+\alpha)\) have the same Mather measures, critical subsolutions,
contact sets and Aubry set. Second, if \(f\) is a positive function, the
fields \(Y\) and \(fY\) have the same oriented orbits, traversed with
different speeds.

\begin{proposition}[Covariance]\label{prop:covariance}
Let \(f\in C^1(U)\) be positive and let \(c'=c(K,fY,\ell)\). Then \(K\) is
invariant under the flow of \(fY\), and:
\begin{enumerate}[label=\textnormal{(\roman*)}]
\item \(\mu\mapsto\nu=f^{-1}\mu/\int f^{-1}\dd\mu\) is a bijection
\(\cM_Y(K)\to\cM_{fY}(K)\), and \(c'\) is the unique constant such that
\(c\bigl(K,Y,(\ell-c')/f\bigr)=0\);
\item \(\nu\) is a Mather measure of \((K,fY,\ell)\) if and only if \(\mu\) is a
Mather measure of \((K,Y,(\ell-c')/f)\); the Mather sets coincide;
\item the sets of critical subsolutions and the contact sets of
\((K,fY,\ell)\) and of \((K,Y,(\ell-c')/f)\) coincide; in particular one of
them admits a critical subsolution if and only if the other does, and in
that case their Aubry sets coincide.
\end{enumerate}
Moreover, \((K,-Y,\ell)\) has the same critical value and Mather measures as
\((K,Y,\ell)\), \(\chi\mapsto-\chi\) exchanges their critical subsolutions, and,
when these exist, their Aubry sets coincide.
\end{proposition}

\begin{proof}
Let \(\Phi_\sigma\) be the flow of \(Y\) and \(v\in K\). Since \(K\) is
invariant and compact, \(\sigma\mapsto\Phi_\sigma(v)\) is defined on \(\R\),
and \(f\) is bounded between two positive constants \(f_-\le f\le f_+\) on
\(K\). The function \(\theta_v(\sigma)=\int_0^\sigma f(\Phi_{\sigma'}(v))^{-1}\dd\sigma'\)
is therefore an increasing \(C^1\) diffeomorphism of \(\R\) onto \(\R\), and,
denoting by \(\sigma_v\) its inverse, the curve
\(s\mapsto\Phi_{\sigma_v(s)}(v)\) satisfies
\(\frac{d}{ds}\Phi_{\sigma_v(s)}(v)=\sigma_v'(s)\,Y=f\,Y\) at
\(\Phi_{\sigma_v(s)}(v)\). By uniqueness it is the orbit of \(v\) under the flow
of \(fY\), which is thus defined for all times and stays in \(K\). In
particular, along an orbit, the change of variables \(ds=f^{-1}\dd\sigma\)
gives, for every continuous function \(h\) on \(K\) and every interval,
\begin{equation}\label{eq:clockchange}
\int_{s_1}^{s_2}h\bigl(\Phi_{\sigma_v(s)}(v)\bigr)\dd s
=\int_{\sigma_v(s_1)}^{\sigma_v(s_2)}\frac{h}{f}\bigl(\Phi_\sigma(v)\bigr)\dd\sigma .
\end{equation}

We now prove the first assertion. Given \(\mu\in\cM_Y(K)\), the measure
\(\nu=f^{-1}\mu/\int f^{-1}\dd\mu\) is a probability measure on \(K\), and for
every \(\varphi\in C^1(M)\)
\[
\int_Kd\varphi(fY)\dd\nu
=\frac{1}{\int f^{-1}\dd\mu}\int_Kf\,d\varphi(Y)\,f^{-1}\dd\mu
=\frac{1}{\int f^{-1}\dd\mu}\int_Kd\varphi(Y)\dd\mu=0,
\]
so that \(\nu\in\cM_{fY}(K)\) by \Cref{lem:infinv}, applied to the field
\(fY\). The same computation, applied to the field \(f^{-1}(fY)=Y\), shows that
\(\nu\mapsto f\nu/\int f\dd\nu\) maps \(\cM_{fY}(K)\) into \(\cM_Y(K)\), and the
two maps are inverse to each other, since
\(f\cdot f^{-1}\mu/\int f\,f^{-1}\dd\mu=\mu\) and symmetrically. For
corresponding \(\mu\) and \(\nu\),
\[
\int_K\ell\dd\nu-c'=\frac{\int_K(\ell-c')f^{-1}\dd\mu}{\int_Kf^{-1}\dd\mu}.
\]
By definition of \(c'\) the left-hand side is nonnegative for every \(\nu\) and
vanishes exactly when \(\nu\) is a Mather measure of \((K,fY,\ell)\); since
Mather measures exist and the denominator is positive, the function
\(\Theta(\alpha):=c\bigl(K,Y,(\ell-\alpha)/f\bigr)=\min_{\mu}\int(\ell-\alpha)f^{-1}\dd\mu\)
satisfies \(\Theta(c')=0\). If \(\alpha_1<\alpha_2\), then
\((\ell-\alpha_2)/f\le(\ell-\alpha_1)/f-(\alpha_2-\alpha_1)/f_+\) on \(K\), whence
\(\Theta(\alpha_2)\le\Theta(\alpha_1)-(\alpha_2-\alpha_1)/f_+<\Theta(\alpha_1)\).
The function \(\Theta\) is thus strictly decreasing, and \(c'\) is the only
zero of \(\Theta\).

For the second assertion, recall that \(c\bigl(K,Y,(\ell-c')/f\bigr)=0\). The
displayed identity shows that \(\nu\) is a Mather measure of \((K,fY,\ell)\) if
and only if \(\int(\ell-c')f^{-1}\dd\mu=0\), that is, if and only if \(\mu\) is a
Mather measure of \((K,Y,(\ell-c')/f)\). Since \(\mu\) and \(\nu\) differ by the
continuous positive density \(f^{-1}/\int f^{-1}\dd\mu\), they have the same
support, and the two Mather sets, which are the closures of the unions of
these supports, coincide.

For the third assertion, let \(\chi\in C^1(M)\). Since \(d\chi(fY)=f\,d\chi(Y)\),
\[
\ell-c'-d\chi(fY)=f\Bigl(\frac{\ell-c'}{f}-0-d\chi(Y)\Bigr)\qquad\text{on }K,
\]
and, \(f\) being positive, the two functions \(g_\chi\) computed for
\((K,fY,\ell)\), whose critical value is \(c'\), and for \((K,Y,(\ell-c')/f)\),
whose critical value is \(0\), have the same sign at every point and the same
zero set. Hence the two problems have the same critical subsolutions and the
same contact sets; in particular one of them has a critical subsolution if
and only if the other does. In that case \Cref{thm:structure}(ii) expresses
each Aubry set as the intersection of the contact sets of all critical
subsolutions, and the two Aubry sets coincide. Equivalently, by
\eqref{eq:clockchange} with \(h=\ell-c'\), an orbit is calibrated for
\((K,fY,\ell)\) if and only if it is calibrated, with the same \(\chi\), for
\((K,Y,(\ell-c')/f)\).

It remains to consider \(-Y\). Its flow is \(\Phi_{-\sigma}\), so a measure is
invariant under the flow of \(-Y\) if and only if it is invariant under
\(\Phi_\sigma\); the two problems have the same invariant measures, hence the
same critical value \(c\) and the same Mather measures. For \(\chi\in C^1(M)\),
\(\ell-c-d\chi(-Y)=\ell-c-d(-\chi)(Y)\), so \(\chi\) is a critical
subsolution for \(-Y\) if and only if \(-\chi\) is one for \(Y\), with the same
contact set. When critical subsolutions exist, \Cref{thm:structure}(ii) shows
again that the two Aubry sets, being the intersections of the same family of
contact sets, coincide.
\end{proof}

We now apply \Cref{prop:covariance} to the two clocks of the cubic model. By
\eqref{eq:YDX}, on \(\Sph^2\cap\set{D>0}\) one has \(\widehat X_s=fY_s\) with
\(f=1/D\), and a compact subset of \(\Sigma_S\cap\set{D>0}\) is invariant
under \(Y_s\) if and only if it is invariant under \(\widehat X_s\).

\begin{corollary}[Polynomial and spherical clocks]\label{cor:clocks}
Let \(K\subset\Sigma_S\cap\set{D>0}\) be compact and invariant, and
\(\ell\in C(K)\).
\begin{enumerate}[label=\textnormal{(\roman*)}]
\item The critical value of \((K,\widehat X_s,\ell)\) is
\[
c(K,\widehat X_s,\ell)=\min_{\mu\in\cM_{Y_s}(K)}\frac{\int_K\ell\,D\dd\mu}{\int_KD\dd\mu},
\]
and, with \(c'=c(K,\widehat X_s,\ell)\), the problems \((K,\widehat X_s,\ell)\) and
\((K,Y_s,(\ell-c')D)\) have corresponding Mather measures, the same Mather
set, the same critical subsolutions and the same Aubry set.
\item Let \(c_D:=c(K,Y_s,D)\). Then \(c_D>0\),
\(c(K,\widehat X_s,-c_D/D)=-1\), and the problems \((K,Y_s,D)\) and
\((K,\widehat X_s,-c_D/D)\) have corresponding Mather measures, the same
Mather set, the same critical subsolutions and the same Aubry set.
\end{enumerate}
The case \(K\subset\set{D<0}\) reduces to this one through the antipodal map,
by \Cref{prop:antipodal}(i) and the last part of \Cref{prop:covariance}.
\end{corollary}

\begin{proof}
The function \(f=1/D\) is smooth and positive on the open set
\(U=\Sph^2\cap\set{D>0}\), which contains \(K\). The first assertion is
\Cref{prop:covariance} with \(Y=Y_s\), since then \(fY=\widehat X_s\), the
correspondence of measures is \(\nu=D\mu/\int D\dd\mu\), and
\(\int\ell\dd\nu=\int\ell D\dd\mu/\int D\dd\mu\).

For the second assertion, \(c_D=\min_\mu\int D\dd\mu\) is positive because \(D\)
is positive and continuous on the compact set \(K\). Applying the formula of
the first assertion to the cost \(-c_D/D\), we find
\[
c\bigl(K,\widehat X_s,-c_D/D\bigr)
=\min_{\mu\in\cM_{Y_s}(K)}\frac{-c_D}{\int_KD\dd\mu}
=-\frac{c_D}{\min_\mu\int_KD\dd\mu}=-1 .
\]
By the first assertion with \(c'=-1\), the problem \((K,\widehat X_s,-c_D/D)\)
has the same structure as \((K,Y_s,(-c_D/D+1)D)=(K,Y_s,D-c_D)\), which differs
from \((K,Y_s,D)\) by the constant \(-c_D\) and therefore has the same Mather
measures, critical subsolutions, contact sets and Aubry set.
\end{proof}

\begin{remark}[Intrinsic and non-intrinsic costs]\label{rem:intrinsic}
\Cref{cor:clocks}(ii) shows that the Mather problem for the clock density
\(\abs{D}\) with the polynomial clock is equivalent to maximizing the average,
along the spherical clock, of \(d\sigma/d\tau=1/\abs{D}\). It
therefore measures the desingularization rather than the dynamics. Costs
defined in terms of the physical dynamics, such as the radial exponent
\(\rho\) with the clock \(\tau\), lead to intrinsic structures, and
\Cref{cor:clocks}(i) allows one to compute them with the polynomial field
\(Y_s\), which is globally defined on \(\Sph^2\).
\end{remark}

% =================================================================
\section{Aubry--Mather structures on the abnormal locus}\label{sec:periodic}
% =================================================================

%\textcolor{red}{\bf RIPRENDERE DA QUI : Chiedere a Claude di dettagliare le dimostrazioni}

\subsection{Periodic components, selection and transverse stability}

\begin{proposition}\label{thm:periodic}
Let \(\Gamma\) be a periodic orbit of \(Y\), not reduced to a point, with
minimal period \(T_\Gamma\) and parametrization \(\gamma\), and let
\(\ell\in C(\Gamma)\). Then:
\begin{enumerate}[label=\textnormal{(\roman*)}]
\item \(\cM_Y(\Gamma)=\set{\mu_\Gamma}\), where
\(\int h\dd\mu_\Gamma=T_\Gamma^{-1}\int_0^{T_\Gamma}h(\gamma(s))\dd s\) for
\(h\in C(\Gamma)\), and \(c(\Gamma,Y,\ell)=\ov\ell_\Gamma:=\int\ell\dd\mu_\Gamma\);
\item the function
\begin{equation}\label{eq:chiGamma}
\chi_\Gamma(\gamma(t)):=\int_0^t\bigl(\ell(\gamma(s))-\ov\ell_\Gamma\bigr)\dd s
\end{equation}
is well defined on \(\Gamma\), solves the critical equation
\(d\chi_\Gamma(Y)=\ell-\ov\ell_\Gamma\) along \(\Gamma\), and admits a \(C^1\)
extension to \(M\), which is a critical subsolution with
\(\cZ(\chi_\Gamma)=\Gamma\); moreover, \(\ov\ell_\Gamma\) is the only constant
for which the equation \(d\chi(Y)=\ell-c\) has a continuous solution on
\(\Gamma\) differentiable along \(Y\);
\item \(\wt\cA_s=\wt\cM_s=\Gamma\), and every critical subsolution coincides with
\(\chi_\Gamma\) on \(\Gamma\) up to an additive constant.
\end{enumerate}
\end{proposition}

\begin{proof}
Since \(\gamma\) is a non-stationary solution of \(\gamma'=Y(\gamma)\) with
minimal period \(T_\Gamma\), it induces a continuous bijection
\(\psi:\R/T_\Gamma\mathbb Z\to\Gamma\), \(\psi([t])=\gamma(t)\), which is a
homeomorphism because its domain is compact; in particular \(\gamma(t)=\gamma(t')\)
if and only if \(t-t'\in T_\Gamma\mathbb Z\). In the coordinate \([t]\) the flow of
\(Y\) acts on \(\Gamma\) by translations, \(\Phi_\tau(\gamma(t))=\gamma(t+\tau)\).
Let \(\mu\in\cM_Y(\Gamma)\) and \(h\in C(\Gamma)\). By invariance,
\(\int h\dd\mu=\int h\circ\Phi_\tau\dd\mu\) for every \(\tau\), and averaging
over \(\tau\in[0,T_\Gamma]\), Fubini's theorem gives
\[
\int_\Gamma h\dd\mu=\int_\Gamma\Bigl(\frac1{T_\Gamma}\int_0^{T_\Gamma}h(\Phi_\tau x)\dd\tau\Bigr)\dd\mu(x).
\]
For \(x=\gamma(t)\) the inner integral equals
\(T_\Gamma^{-1}\int_t^{t+T_\Gamma}h(\gamma(s))\dd s=T_\Gamma^{-1}\int_0^{T_\Gamma}h(\gamma(s))\dd s\),
by periodicity, and does not depend on \(x\); since \(\mu\) is a probability
measure, \(\int h\dd\mu=\int h\dd\mu_\Gamma\). Conversely \(\mu_\Gamma\) is
invariant, because \(\int h\circ\Phi_\tau\dd\mu_\Gamma
=T_\Gamma^{-1}\int_0^{T_\Gamma}h(\gamma(s+\tau))\dd s=\int h\dd\mu_\Gamma\), again by
periodicity. Hence \(\cM_Y(\Gamma)=\set{\mu_\Gamma}\), and
\(c(\Gamma,Y,\ell)=\int\ell\dd\mu_\Gamma=\ov\ell_\Gamma\) by definition of the
critical value.

We turn to the corrector. By definition of \(\ov\ell_\Gamma\),
\(\int_0^{T_\Gamma}(\ell(\gamma(s))-\ov\ell_\Gamma)\dd s=0\), so the function
\(t\mapsto\int_0^t(\ell(\gamma(s))-\ov\ell_\Gamma)\dd s\) is
\(T_\Gamma\)-periodic, since its increment over any interval of length
\(T_\Gamma\) equals the integral over one period of a \(T_\Gamma\)-periodic
function. As \(\gamma(t)=\gamma(t')\) exactly when \(t-t'\in T_\Gamma\mathbb Z\),
formula \eqref{eq:chiGamma} defines a function on \(\Gamma\). The integrand
being continuous, \(t\mapsto\chi_\Gamma(\gamma(t))\) is of class \(C^1\), with
\(\frac{d}{dt}\chi_\Gamma(\gamma(t))=\ell(\gamma(t))-\ov\ell_\Gamma\), which is the
critical equation along \(\Gamma\). To extend \(\chi_\Gamma\), observe that
\(\gamma'=Y(\gamma)\) with \(Y\) of class \(C^1\), so that \(\gamma'\) is \(C^1\)
and \(\gamma\) is of class \(C^2\); moreover \(\gamma'\) never vanishes, and
\(\psi\) is a homeomorphism, so \(\Gamma\) is a compact embedded submanifold of
class \(C^2\), and \(\chi_\Gamma=(\chi_\Gamma\circ\gamma)\circ\psi^{-1}\) is of class
\(C^1\) on it. Fix a smooth Riemannian metric on \(M\). For a compact embedded
submanifold of class \(C^2\), the nearest-point projection \(\pi\) is defined
and of class \(C^1\) on a tubular neighbourhood \(V\) of \(\Gamma\)
\cite{Foote84}, and \(\pi|_\Gamma=\mathrm{id}\). Let \(\theta\in C^\infty_c(V)\)
be equal to one on a neighbourhood of \(\Gamma\); then
\(\theta\cdot(\chi_\Gamma\circ\pi)\), extended by zero outside \(V\), is a \(C^1\)
function on \(M\) which coincides with \(\chi_\Gamma\) on \(\Gamma\), and which we
still denote by \(\chi_\Gamma\). Since \(Y(\gamma(t))=\gamma'(t)\) is tangent to
\(\Gamma\), the derivative \(d\chi_\Gamma(Y)\) at \(\gamma(t)\) only depends on
the restriction of \(\chi_\Gamma\) to \(\Gamma\), and equals
\(\frac{d}{dt}\chi_\Gamma(\gamma(t))=\ell(\gamma(t))-\ov\ell_\Gamma\). Therefore
\(g_{\chi_\Gamma}=\ell-\ov\ell_\Gamma-d\chi_\Gamma(Y)\) vanishes identically on
\(K=\Gamma\): the extension is a critical subsolution, and its contact set is
the whole of \(\Gamma\). Finally, if \(\chi\) is continuous on \(\Gamma\),
differentiable along \(Y\), and satisfies \(d\chi(Y)=\ell-c\) on \(\Gamma\) for some
constant \(c\), then integrating
\(\frac{d}{dt}\chi(\gamma(t))=\ell(\gamma(t))-c\) over \([0,T_\Gamma]\) gives
\(0=\chi(\gamma(T_\Gamma))-\chi(\gamma(0))=\int_0^{T_\Gamma}\ell(\gamma(s))\dd s-cT_\Gamma\),
that is, \(c=\ov\ell_\Gamma\).

It remains to identify the Aubry and Mather sets. The set of critical
subsolutions is nonempty, since it contains \(\chi_\Gamma\), so
\Cref{thm:structure} applies. The measure \(\mu_\Gamma\) is the only invariant
probability measure, hence a Mather measure, and its support is \(\Gamma\),
because it is the image under \(\psi\) of the normalized Lebesgue measure,
which charges every nonempty open subset of \(\R/T_\Gamma\mathbb Z\). By
\Cref{thm:structure}(iii) we obtain
\(\Gamma=\supp\mu_\Gamma\subset\wt\cM_s\subset\wt\cA_s\subset K=\Gamma\), so that
\(\wt\cA_s=\wt\cM_s=\Gamma\). If \(\chi\) is any critical subsolution, then, since
\(\gamma(0)\in\wt\cA_s\), \Cref{thm:structure}(iv) gives
\(\chi(\gamma(t))-\chi(\gamma(0))=\int_0^t(\ell(\gamma(s))-\ov\ell_\Gamma)\dd s
=\chi_\Gamma(\gamma(t))\) for every \(t\), so that \(\chi=\chi_\Gamma+\chi(\gamma(0))\)
on \(\Gamma\).
\end{proof}

\begin{remark}
\Cref{thm:periodic} is the weak KAM picture on a single closed orbit; the
solvability condition in (ii) is the elementary case of the Gottschalk--Hedlund
theorem on coboundaries \cite{GH55}. On its own it contains no selection,
since \(\Gamma\) carries a unique invariant measure. Selection appears as soon
as \(\Gamma\) competes with other invariant objects (\Cref{thm:selection}).
\end{remark}

Applying \Cref{thm:periodic} to a regular periodic abnormal component with
the two costs of \Cref{ex:settings} we obtain:

\begin{corollary}\label{cor:radial}
Let \(\Gamma\) be a regular periodic abnormal component.
\begin{enumerate}[label=\textnormal{(\alph*)}]
\item The period of \(\Gamma\) as an orbit of \(\widehat X_s\) is
\(T^\tau_\Gamma=\int_0^{T^\sigma_\Gamma}\abs{D(\gamma(\sigma))}\dd\sigma\), where
\(\gamma\) is a parametrization of \(\Gamma\) by \(Y_s\) of period
\(T^\sigma_\Gamma\), and the critical value of \((\Gamma,Y_s,\abs{D})\) is
\(T^\tau_\Gamma/T^\sigma_\Gamma\).
\item The critical value of \((\Gamma,\widehat X_s,\rho)\) is the mean radial
exponent \(\ov\rho_\Gamma\) of \Cref{prop:lift}. Every critical subsolution
\(\chi\) of \((\Gamma,\widehat X_s,\rho)\) satisfies, along every abnormal
trajectory \(v=r\,w\) of \(X_s\) on the cone over \(\Gamma\),
\[
\log r(\tau)=\log r(0)+\ov\rho_\Gamma\,\tau+\chi(w(\tau))-\chi(w(0))
\qquad\text{for all }\tau .
\]
Thus \(\chi\) restricted to \(\Gamma\) is, up to an additive constant, the
periodic part of \(\log\abs{v}\); the sign of \(\ov\rho_\Gamma\) decides between
convergence to the origin, blow-up in finite time and periodicity as in
\Cref{prop:lift}(a)--(c); and the admissibility condition of
\Cref{prop:lift}(d) reads \(r_0\le m_\Gamma^{-1/3}e^{-\osc_\Gamma\chi}\), where
\(\osc_\Gamma\chi=\max_\Gamma\chi-\min_\Gamma\chi\) does not depend on the starting
point on \(\Gamma\).
\end{enumerate}
\end{corollary}

\begin{proof}
Let \(\gamma\) be a parametrization of \(\Gamma\) by \(Y_s\), with minimal period
\(T^\sigma_\Gamma\). Since \(\Gamma\) is connected and \(D\) does not vanish on
it, \(D\) has a constant sign \(\epsilon\in\set{\pm1}\) on \(\Gamma\), and the
function \(\tau(\sigma):=\int_0^\sigma D(\gamma(\sigma'))\dd\sigma'\) is a \(C^1\)
diffeomorphism of \(\R\), increasing if \(\epsilon=1\) and decreasing if
\(\epsilon=-1\), with \(\tau(\sigma+T^\sigma_\Gamma)=\tau(\sigma)+\epsilon T\), where
\(T:=\int_0^{T^\sigma_\Gamma}\abs{D(\gamma(\sigma))}\dd\sigma>0\). Denoting by
\(\sigma(\tau)\) its inverse and setting \(w_0(\tau):=\gamma(\sigma(\tau))\), we
obtain from \eqref{eq:YDX}
\[
w_0'(\tau)=\sigma'(\tau)\,Y_s\bigl(w_0(\tau)\bigr)
=\frac{Y_s(w_0(\tau))}{D(w_0(\tau))}=\widehat X_s\bigl(w_0(\tau)\bigr),
\]
so that \(w_0\) is a parametrization of \(\Gamma\) by \(\widehat X_s\). Moreover
\(w_0(\tau)=w_0(\tau')\) if and only if \(\sigma(\tau)-\sigma(\tau')\in
T^\sigma_\Gamma\mathbb Z\), that is, if and only if \(\tau-\tau'\in T\mathbb Z\);
hence \(\Gamma\) is a periodic orbit of \(\widehat X_s\) of minimal period
\(T=T^\tau_\Gamma\). By \Cref{thm:periodic}(i), applied to the field \(Y_s\) and
the cost \(\abs{D}\), the critical value of \((\Gamma,Y_s,\abs{D})\) is
\(\frac1{T^\sigma_\Gamma}\int_0^{T^\sigma_\Gamma}\abs{D(\gamma(\sigma))}\dd\sigma
=T^\tau_\Gamma/T^\sigma_\Gamma\), which proves the first assertion.

For the second assertion we apply \Cref{thm:periodic} to the field
\(\widehat X_s\), which is smooth on the open set \(U=\Sph^2\cap\set{D\neq0}\)
containing \(\Gamma\), and to the cost \(\rho=\kappa/D\), which is continuous on
\(\Gamma\). The critical value of \((\Gamma,\widehat X_s,\rho)\) is the average of
\(\rho\) over one period of any parametrization of \(\Gamma\) by \(\widehat X_s\),
which is \(\ov\rho_\Gamma\) as defined in \Cref{prop:lift}; this average does
not depend on the parametrization, since two of them differ by a translation
of the time variable. Now let \(v\) be an abnormal trajectory of \(X_s\) with
\(v(0)\) on the cone over \(\Gamma\), and set \(r_0=\abs{v(0)}\),
\(w(0)=v(0)/r_0\in\Gamma\). Let \(w\) be the trajectory of \(\widehat X_s\) through
\(w(0)\), which is defined on \(\R\) and runs along \(\Gamma\). By
\Cref{prop:lift} and uniqueness of solutions, \(v\) coincides with the curve
\(r(\tau(t))\,w(\tau(t))\) constructed there, and
\(\log r(\tau)=\log r_0+\ov\rho_\Gamma\tau+\vartheta(\tau)\) with
\(\vartheta(\tau)=\int_0^\tau(\rho(w(s))-\ov\rho_\Gamma)\dd s\). On the other hand,
let \(\chi\) be a critical subsolution. By \Cref{thm:periodic}(iii), \(\chi\)
coincides on \(\Gamma\) with the corrector \eqref{eq:chiGamma} built from the
parametrization \(w\), up to an additive constant, so that
\[
\chi(w(\tau))-\chi(w(0))=\int_0^\tau\bigl(\rho(w(s))-\ov\rho_\Gamma\bigr)\dd s
=\vartheta(\tau)\qquad\text{for all }\tau\in\R .
\]
Substituting this identity in the expression of \(\log r\) gives the displayed
formula. In particular \(\tau\mapsto\log r(\tau)-\ov\rho_\Gamma\tau\) equals
\(\chi(w(\tau))\) plus the constant \(\log r_0-\chi(w(0))\), so that, along the
trajectory, \(\chi|_\Gamma\) is the periodic part of \(\log\abs{v}\) up to an
additive constant. The trichotomy follows from \Cref{prop:lift}(a)--(c),
whose hypotheses only involve the sign of \(\ov\rho_\Gamma\). Finally, since
\(w\) maps \([0,T^\tau_\Gamma]\) onto \(\Gamma\), the function \(\vartheta\) on one
period takes exactly the values of \(\chi-\chi(w(0))\) on \(\Gamma\), and
therefore \(\osc\vartheta=\osc_\Gamma\chi\). The quantity \(\osc_\Gamma\chi\)
depends neither on the additive constant in \(\chi\) nor on the point \(w(0)\),
and the condition of \Cref{prop:lift}(d) becomes
\(r_0\le m_\Gamma^{-1/3}e^{-\osc_\Gamma\chi}\).
\end{proof}

\begin{theorem}[Selection]\label{thm:selection}
Let \(K=\Gamma_1\cup\dots\cup\Gamma_m\cup E\), where the \(\Gamma_i\) are pairwise
disjoint periodic orbits of \(Y\), not reduced to points, and \(E\) is a
finite set of equilibria of \(Y\) disjoint from them. Let \(\ell\in C(K)\) and
denote by \(\ov\ell_{\Gamma_i}\) the average of \(\ell\) with respect to the
invariant probability measure \(\mu_{\Gamma_i}\) of \Cref{thm:periodic}. Then
\(\cS\neq\emptyset\),
\[
c=c(K,Y,\ell)=\min\Bigl(\set{\ov\ell_{\Gamma_i}:\ 1\le i\le m}\cup\set{\ell(e):\ e\in E}\Bigr),
\]
the Mather measures are the convex combinations of the measures
\(\mu_{\Gamma_i}\) with \(\ov\ell_{\Gamma_i}=c\) and of the Dirac masses \(\delta_e\)
with \(\ell(e)=c\), and \(\wt\cA_s=\wt\cM_s\) is the union of the orbits
\(\Gamma_i\) with \(\ov\ell_{\Gamma_i}=c\) and of the equilibria \(e\in E\) with
\(\ell(e)=c\).
\end{theorem}

\begin{proof}
The set \(K\) is the disjoint union of the finitely many compact sets
\(\Gamma_1,\dots,\Gamma_m\) and \(\set{e}\), \(e\in E\), each of which is
invariant under the flow \(\Phi_t\) of \(Y\); in particular \(K\) is compact and
invariant. Let \(\mu\in\cM_Y(K)\). If \(P\) is one of these pieces and \(A\) a
Borel set, then \(\Phi_t^{-1}(A\cap P)=\Phi_t^{-1}(A)\cap P\) by invariance of
\(P\), so that the restriction \(\mu|_P\) is again invariant. Writing
\(a_i=\mu(\Gamma_i)\) and \(b_e=\mu(\set{e})\), we have \(\sum_ia_i+\sum_eb_e=1\);
whenever \(a_i>0\), the normalized restriction \(a_i^{-1}\mu|_{\Gamma_i}\) is an
invariant probability measure on \(\Gamma_i\), hence equal to \(\mu_{\Gamma_i}\)
by \Cref{thm:periodic}(i). Therefore
\[
\mu=\sum_{i=1}^ma_i\,\mu_{\Gamma_i}+\sum_{e\in E}b_e\,\delta_e ,
\]
and conversely every such convex combination is invariant, each of its terms
being invariant. For such \(\mu\),
\[
\int_K\ell\dd\mu=\sum_ia_i\,\ov\ell_{\Gamma_i}+\sum_eb_e\,\ell(e)\ge c_0,
\qquad
c_0:=\min\Bigl(\set{\ov\ell_{\Gamma_i}}\cup\set{\ell(e)}\Bigr),
\]
with equality if and only if all the weights \(a_i\) and \(b_e\) attached to
pieces whose average is larger than \(c_0\) vanish. Since the minimum \(c_0\) is
attained by a single \(\mu_{\Gamma_i}\) or \(\delta_e\), we conclude that
\(c=c_0\) and that the Mather measures are exactly the convex combinations
described in the statement. As \(\supp\mu_{\Gamma_i}=\Gamma_i\), by the proof of
\Cref{thm:periodic}, and \(\supp\delta_e=\set{e}\), the Mather set \(\wt\cM_s\),
which is the closure of the union of the supports of Mather measures, is the
union of the minimizing orbits and equilibria; this union is finite, hence
closed.

We now construct a critical subsolution. The pieces of \(K\) being compact and
pairwise disjoint, we can choose open sets \(V_1,\dots,V_m\) with pairwise
disjoint closures, such that \(\Gamma_i\subset V_i\) and \(\ov V_i\) does not meet
\(K\setminus\Gamma_i\). For each \(i\), \Cref{thm:periodic}(ii), applied to the
periodic orbit \(\Gamma_i\) and to \(\ell|_{\Gamma_i}\), provides a \(C^1\) function
\(\chi_i\) on \(M\) with \(d\chi_i(Y)=\ell-\ov\ell_{\Gamma_i}\) on \(\Gamma_i\);
multiplying it by a cutoff function supported in \(V_i\) and equal to one near
\(\Gamma_i\), we may assume that \(\supp\chi_i\subset V_i\), which does not change
\(d\chi_i\) on \(\Gamma_i\). Set \(\chi:=\chi_1+\dots+\chi_m\in C^1(M)\). On
\(\Gamma_i\), all the functions \(\chi_j\) with \(j\neq i\) vanish in a
neighbourhood, so that
\[
g_\chi=\ell-c-d\chi(Y)=\ell-c-\bigl(\ell-\ov\ell_{\Gamma_i}\bigr)=\ov\ell_{\Gamma_i}-c\ge0
\qquad\text{on }\Gamma_i .
\]
At an equilibrium \(e\in E\) we have \(Y(e)=0\), hence \(d\chi(Y)(e)=0\) and
\(g_\chi(e)=\ell(e)-c\ge0\). Thus \(\chi\in\cS\), and the contact set
\(\cZ(\chi)\), on which \(g_\chi\) vanishes, is the union of the orbits
\(\Gamma_i\) with \(\ov\ell_{\Gamma_i}=c\) and of the equilibria with \(\ell(e)=c\),
that is, \(\cZ(\chi)=\wt\cM_s\). Since \(\cS\neq\emptyset\), \Cref{thm:structure}
gives \(\wt\cM_s\subset\wt\cA_s\subset\cZ(\chi)\), and therefore
\(\wt\cA_s=\wt\cM_s\).
\end{proof}

\begin{corollary}\label{cor:pairs}
Let \(\Gamma\) be a regular periodic abnormal component. Then \(-\Gamma\) is a
regular periodic abnormal component, and, for every \(\ell\in C(\Gamma\cup-\Gamma)\)
and for both choices \(Y=\widehat X_s\) and \(Y=Y_s\),
\[
\ov\ell_{-\Gamma}=\ov{(\ell\circ\mathsf a)}_\Gamma ,
\]
where the averages are taken with respect to the invariant probability
measures of the flow of \(Y\) on \(-\Gamma\) and on \(\Gamma\). In particular
\(\ov\ell_{-\Gamma}=\ov\ell_\Gamma\) if \(\ell\) is even (\(\ell\circ\mathsf a=\ell\)),
as for \(\rho\) and \(\abs{D}\), and \(\ov\ell_{-\Gamma}=-\ov\ell_\Gamma\) if \(\ell\) is
odd, as for \(\kappa\). Consequently, if \(\ell\) is even and \(K\) is as in
\Cref{thm:selection} with both \(\Gamma\) and \(-\Gamma\) among its periodic
orbits, then \(\Gamma\subset\wt\cA_s\) if and only if \(-\Gamma\subset\wt\cA_s\).
\end{corollary}

\begin{proof}
Since \(D''\) is odd, \(\Sigma_S\) is invariant under \(\mathsf a\), and \(\mathsf a\),
being a homeomorphism of \(\Sph^2\) preserving \(\Sigma_S\), maps connected
components of \(\Sigma_S\) onto connected components; thus \(-\Gamma\) is a
connected component of \(\Sigma_S\). The gradient of the odd polynomial \(D''\)
is even, so \(\nabla D''(-v)=\nabla D''(v)\neq0\) for \(v\in\Gamma\), and every
point of \(-\Gamma\) is regular. Moreover \(D(-v)=-D(v)\neq0\) and, by
\Cref{prop:antipodal}(i), \(Y_s(-v)=Y_s(v)\neq0\) for \(v\in\Gamma\). Hence
\(-\Gamma\) is a regular periodic abnormal component, and it is a periodic orbit
of both \(Y_s\) and \(\widehat X_s\).

Consider first \(Y=\widehat X_s\), and let \(w\) be a parametrization of \(\Gamma\)
by \(\widehat X_s\), of minimal period \(T\). By \Cref{prop:antipodal}(i), the
curve \(-w\) is a trajectory of \(\widehat X_s\) running along \(-\Gamma\); since
\(\mathsf a\) is a bijection, \(-w(\tau)=-w(\tau')\) if and only if
\(w(\tau)=w(\tau')\), so that \(-w\) has the same minimal period \(T\). By
\Cref{thm:periodic}(i),
\[
\ov\ell_{-\Gamma}=\frac1T\int_0^T\ell\bigl(-w(\tau)\bigr)\dd\tau
=\frac1T\int_0^T(\ell\circ\mathsf a)\bigl(w(\tau)\bigr)\dd\tau
=\ov{(\ell\circ\mathsf a)}_\Gamma .
\]
Consider now \(Y=Y_s\), and let \(\gamma\) be a parametrization of \(\Gamma\) by
\(Y_s\), of minimal period \(T\). By \Cref{prop:antipodal}(i), the curve
\(\tilde\gamma(\sigma):=-\gamma(-\sigma)\) is a trajectory of \(Y_s\) running along
\(-\Gamma\), and, as before, it has minimal period \(T\). Using the change of
variables \(s=-\sigma\) and the \(T\)-periodicity of \(\gamma\),
\[
\ov\ell_{-\Gamma}=\frac1T\int_0^T\ell\bigl(-\gamma(-\sigma)\bigr)\dd\sigma
=\frac1T\int_{-T}^0(\ell\circ\mathsf a)\bigl(\gamma(s)\bigr)\dd s
=\frac1T\int_0^T(\ell\circ\mathsf a)\bigl(\gamma(s)\bigr)\dd s
=\ov{(\ell\circ\mathsf a)}_\Gamma .
\]
If \(\ell\) is even, then \(\ell\circ\mathsf a=\ell\) and the two averages
coincide; if \(\ell\) is odd, then \(\ell\circ\mathsf a=-\ell\) and they are
opposite. The functions \(\abs{D}\) and \(\rho\) are even, because \(D\) is odd
and \(\rho(-w)=\pair{X_s(-w)}{-w}=\pair{-X_s(w)}{-w}=\rho(w)\), while
\(\kappa(-v)=\pair{Z_s(-v)}{-v}=-\kappa(v)\) since \(Z_s\) is even.

Finally, let \(\ell\) be even and \(K\) as in \Cref{thm:selection}, with both
\(\Gamma\) and \(-\Gamma\) among its periodic orbits. By \Cref{thm:selection},
\(\Gamma\subset\wt\cA_s\) if and only if \(\ov\ell_\Gamma=c(K,Y,\ell)\), and
\(-\Gamma\subset\wt\cA_s\) if and only if \(\ov\ell_{-\Gamma}=c(K,Y,\ell)\). These
two conditions are equivalent, because \(\ov\ell_{-\Gamma}=\ov\ell_\Gamma\).
\end{proof}

We now show that the averages of the algebraic cost \(k=\cK-5\kappa\) of
\Cref{thm:multiplier} govern the transverse stability of the abnormal locus
near any compact invariant set, not only near a single periodic component.

\begin{proposition}\label{prop:transverse}
Let \(K\subset\Sigma_S\) be compact and invariant under \(Y_s\), and let
\[
\lambda^+(K):=\max_{\mu\in\cM_{Y_s}(K)}\int_Kk\dd\mu=-c(K,Y_s,-k).
\]
If \(\lambda^+(K)<0\), there are a neighbourhood \(V\) of \(K\), a time \(T>0\)
and a constant \(a>0\) such that \(\abs{D''(\Phi_T(v))}\le e^{-aT}\abs{D''(v)}\) for
every \(v\in V\); consequently \(\abs{D''(\Phi_{nT}(v))}\le e^{-naT}\abs{D''(v)}\)
as long as \(\Phi_{jT}(v)\in V\) for \(0\le j<n\). The maximizing measures are
the Mather measures of \((K,Y_s,-k)\). If \(K=\Gamma_1\cup\dots\cup\Gamma_m\cup E\) is
as in \Cref{thm:selection}, then
\[
\lambda^+(K)=\max\Bigl(\set{\lambda_{\Gamma_i}/T^\sigma_{\Gamma_i}}\cup\set{k(e):\ e\in E}\Bigr),
\]
and, for a regular equilibrium \(e\), \(k(e)\) is the eigenvalue of the
linearization of \(Y_s\) at \(e\) in the direction transverse to \(\Sigma_S\).
\end{proposition}

\begin{proof}
The identity \(\lambda^+(K)=-c(K,Y_s,-k)\) is the definition of the critical
value applied to the cost \(-k\), and the maximizing measures of \(\int k\dd\mu\)
are the minimizing measures of \(\int(-k)\dd\mu\), that is, the Mather
measures of \((K,Y_s,-k)\). Assume \(\lambda^+(K)<0\) and set
\(a:=-\lambda^+(K)/2>0\). By \Cref{thm:duality}(ii) applied to the cost \(-k\),
\[
\sup_{T>0}\frac1T\min_{v\in K}\int_0^T-k(\Phi_s v)\dd s=c(K,Y_s,-k)=2a,
\]
so there is \(T>0\) with \(\int_0^Tk(\Phi_sv)\dd s<-\tfrac32aT\) for every
\(v\in K\). The function \(v\mapsto\int_0^Tk(\Phi_sv)\dd s\) is continuous on
\(\Sph^2\), the flow being continuous on \([0,T]\times\Sph^2\), so the strict
inequality \(\int_0^Tk(\Phi_sv)\dd s<-aT\) holds on an open neighbourhood \(V\)
of the compact set \(K\). For \(v\in V\), \Cref{cor:sphereinvariance} gives
\[
\abs{D''(\Phi_T(v))}=\abs{D''(v)}\exp\int_0^Tk(\Phi_sv)\dd s\le e^{-aT}\abs{D''(v)},
\]
and the iterated estimate follows by induction. When \(K\) is as in
\Cref{thm:selection}, the invariant measures are the convex combinations of
the measures \(\mu_{\Gamma_i}\) and \(\delta_e\), and \(\int k\dd\mu_{\Gamma_i}=
\lambda_{\Gamma_i}/T^\sigma_{\Gamma_i}\) by \Cref{thm:periodic}(i); the maximum of the
linear function \(\mu\mapsto\int k\dd\mu\) is attained at one of these
extremal measures, which gives the formula. Finally, let \(e\in\Sigma_S\) be a
regular equilibrium. Then \(Y_s(e)=0\), the linearization \(L=DY_s(e)\) leaves
\(T_e\Sigma_S\) invariant because \(\Sigma_S\) is invariant near \(e\), and
differentiating \(dD''(Y_s)=kD''\) at \(e\) gives \(dD''(e)\circ L=k(e)\,dD''(e)\), since
\(D''(e)=0\). Thus \(dD''(e)\), which is nonzero and vanishes on
\(T_e\Sigma_S\), is a left eigenvector of \(L\) with eigenvalue \(k(e)\): this is the
eigenvalue of \(L\) on the quotient \(T_e\Sph^2/T_e\Sigma_S\), that is, in the
direction transverse to \(\Sigma_S\).
\end{proof}

\begin{remark}
\Cref{prop:transverse} shows that the stability of the abnormal locus near a
compact invariant set is decided by a singular Mather problem with the
algebraic cost \(k=\cK-5\kappa\): the least stable invariant objects are the
Mather set of \((K,Y_s,-k)\). The uniformity of the estimate, even when \(K\)
carries many invariant measures, comes from the duality formula of
\Cref{thm:duality}, and is the point where the Aubry--Mather framework of
\Cref{sec:AMstructures} is used in an essential way.
\end{remark}

%\textcolor{red}{\bf RIPRENDERE DA QUI : Chiedere a Claude di dettagliare le dimostrazioni}

\subsection{Heteroclinic cycles and the existence of correctors}

When \(K\) contains orbits which are neither periodic nor stationary,
critical subsolutions need not exist. The obstruction is concentrated on
heteroclinic connections.

\begin{proposition}\label{prop:heteroclinic}
Let \(c=c(K,Y,\ell)\). Let \(e_1,\dots,e_m\in K\) be equilibria and
\(\eta_1,\dots,\eta_m\subset K\) orbits such that \(\eta_j\) is heteroclinic from
\(e_j\) to \(e_{j+1}\) (indices modulo \(m\)), that is,
\(\eta_j(s)\to e_j\) as \(s\to-\infty\) and \(\eta_j(s)\to e_{j+1}\) as
\(s\to+\infty\). Assume that \(\ell(e_j)=c\) for all \(j\) and that the improper
integrals
\[
I_j:=\int_{-\infty}^{+\infty}\bigl(\ell(\eta_j(s))-c\bigr)\dd s
\]
converge. Then:
\begin{enumerate}[label=\textnormal{(\roman*)}]
\item if \(\cS\neq\emptyset\), then \(\sum_jI_j\ge0\); equivalently, if
\(\sum_jI_j<0\) there is no \(C^1\) critical subsolution, and the supremum in
\Cref{thm:duality}(i) is not attained;
\item if \(\cS\neq\emptyset\) and \(\sum_jI_j=0\), the cycle
\(\cC:=\set{e_1,\dots,e_m}\cup\eta_1\cup\dots\cup\eta_m\) is contained in
\(\wt\cA_s\); if moreover \(K=\cC\), then \(\wt\cM_s=\set{e_1,\dots,e_m}\), so
that \(\wt\cM_s\subsetneq\wt\cA_s=\cC\).
\end{enumerate}
\end{proposition}

\begin{proof}
Let \(\chi\in\cS\) and fix \(j\). Since \(\eta_j\) is an orbit contained in \(K\) and
\(g_\chi=\ell-c-d\chi(Y)\ge0\) on \(K\), for \(a<b\) we have
\[
\chi(\eta_j(b))-\chi(\eta_j(a))=\int_a^bd\chi(Y)(\eta_j(s))\dd s
=\int_a^b\bigl(\ell(\eta_j(s))-c\bigr)\dd s-\int_a^bg_\chi(\eta_j(s))\dd s .
\]
As \(a\to-\infty\) and \(b\to+\infty\), the left-hand side tends to
\(\chi(e_{j+1})-\chi(e_j)\) by continuity of \(\chi\), the first integral on the
right tends to \(I_j\) by assumption, and the second one increases, by
monotone convergence, to \(\int_\R g_\chi(\eta_j(s))\dd s\), which is therefore
finite. We obtain
\[
\chi(e_{j+1})-\chi(e_j)=I_j-\varepsilon_j,\qquad
\varepsilon_j:=\int_\R g_\chi(\eta_j(s))\dd s\ge0 .
\]
Summing over \(j\) and using that the left-hand sides telescope to zero,
indices being taken modulo \(m\), we find \(\sum_j\varepsilon_j=\sum_jI_j\), so
\(\sum_jI_j\ge0\). Since \(\chi\in\cS\) if and only if
\(\min_K(\ell-d\chi(Y))\ge c\), and the supremum in \Cref{thm:duality}(i) equals
\(c\), the set \(\cS\) is nonempty exactly when this supremum is attained,
which proves (i).

Assume now \(\sum_jI_j=0\) and let \(\chi\in\cS\). Then \(\sum_j\varepsilon_j=0\) with
\(\varepsilon_j\ge0\), so every \(\varepsilon_j\) vanishes; since \(g_\chi\) is
continuous and nonnegative, \(g_\chi(\eta_j(s))=0\) for every \(s\in\R\). At the
equilibria, \(Y(e_j)=0\) gives \(g_\chi(e_j)=\ell(e_j)-c=0\). Thus
\(\cC\subset\cZ(\chi)\). Every point \(v\in\cC\) has its whole orbit in \(\cC\),
namely \(\set{e_j}\) or \(\eta_j\), so \(g_\chi\) vanishes along this orbit and
integrating the identity \(d\chi(Y)=\ell-c-g_\chi=\ell-c\) along it shows that
the orbit of \(v\) is \(\chi\)-calibrated. As \(\chi\in\cS\) is arbitrary,
\(\cC\subset\wt\cA_s\).

Assume finally \(K=\cC\), and let \(\mu\in\cM_Y(K)\). Each \(\eta_j\) is neither
stationary nor periodic, since it converges to an equilibrium which does not
belong to it, so \(s\mapsto\eta_j(s)\) is injective. The half-open arcs \(A_n:=\eta_j([n,n+1))\), \(n\in\mathbb Z\), are therefore
pairwise disjoint, their union is \(\eta_j\), and \(\Phi_n(A_0)=A_n\). By
invariance \(\mu(A_n)=\mu(A_0)\) for every \(n\), and since
\(\sum_{n\in\mathbb Z}\mu(A_n)=\mu(\eta_j)\le1\), we get \(\mu(A_0)=0\), hence
\(\mu(\eta_j)=0\). Consequently every invariant
probability measure on \(K\) is a convex combination
\(\sum_jb_j\delta_{e_j}\), and conversely every such combination is invariant.
For all of them \(\int\ell\dd\mu=\sum_jb_j\ell(e_j)=c\), so they are all Mather
measures, and the Mather set is the union of their supports,
\(\wt\cM_s=\set{e_1,\dots,e_m}\). Since \(\cC\subset\wt\cA_s\subset K=\cC\), we
conclude that \(\wt\cA_s=\cC\), which strictly contains \(\wt\cM_s\) because the
orbits \(\eta_j\) are not reduced to points.
\end{proof}

\begin{remark}\label{rem:circulation}
\Cref{prop:heteroclinic} identifies the natural circulation condition for
general compact invariant subsets of \(\Sigma_S\): the vanishing of
the circulation of \(\ell-c\) along periodic orbits is automatic for the
minimizing ones and irrelevant for the others, whereas along heteroclinic
cycles the relevant quantity is the total normalized action, which must be
nonnegative. Note also that an \emph{exact} solution of
\(d\chi(Y)=\ell-c\) can be continuous at an equilibrium \(e\) which is the limit
of a regular orbit only if \(\ell(e)=c\), since otherwise \(\chi\) grows
linearly along that orbit. The example of a heteroclinic cycle with \(\sum I_j=0\)
shows that the singular Aubry set may be strictly larger than the singular
Mather set.
\end{remark}

% =================================================================
\section{A numerical example}\label{sec:numerics}
% =================================================================

Let \(b=e_3=(0,0,1)\) and
\begin{equation}\label{eq:exampleH}
H(x,y,z)=
\begin{pmatrix}
-2x^3+2xy^2+2y^2z\\
x^2y-2xz^2+2z^3\\
-2x^3-x^2y-2xyz
\end{pmatrix}.
\end{equation}
Then \(A=\Lie{b}{H}=(2y^2,\,-4xz+6z^2,\,-2xy)\), \(C=\Lie{A}{b}=(0,\,4x-12z,\,0)\), and
\begin{align*}
D&=8y^2(x-3z),\\
D''&=-2\bigl(4x^4z-6x^3z^2-x^2y^3-4x^2y^2z+4xy^2z^2+4y^2z^3\bigr),\\
D'&=-4y\bigl(4x^4y+2x^3y^2+10x^3yz+16x^3z^2-2x^2y^2z-27x^2yz^2-40x^2z^3\\
&\qquad\qquad-2xy^4-8xy^3z-12xy^2z^2+8xz^4+4y^3z^2+24z^5\bigr),\\
\cK&=8y(x-3z)\bigl(3x^2y+8x^2z-8xz^2-8z^3\bigr).
\end{align*}
The identity \(\cL_{Z_s}D''=\cK D''\) of \Cref{thm:darboux} has been checked
symbolically. The set \(\set{D=0}\cap\Sph^2\) is the union of the great circles
\(\set{y=0}\) and \(\set{x=3z}\).

The curve \(\Sigma_S\) is shown in \Cref{fig:example}(a). Its points
\(\pm e_2\) and \(\pm e_3\) are singular (\(\nabla D''=0\)), so
\Cref{prop:antipodal}(iii) does not apply directly; nevertheless the picture
is the expected one: \(\Sigma_S\) consists of an antipodally symmetric
component, on which \(D\) vanishes, and of one pair of regular components
\(\pm\Gamma\) on which \(D\neq0\). A search for the zeros of \(Y_s\) on \(\Sigma_S\)
(least-squares solution of \(D''=0\), \(v\times Z_s(v)=0\), \(\abs{v}=1\) from
\(4000\) random initial guesses) gives the equilibria \(\pm e_1,\pm e_2,\pm e_3\),
where \(Z_s=0\) and \(D=0\), and the two pairs \(\pm p_a\), \(\pm p_b\) listed in
\Cref{tab:equilibria}, where \(D\neq0\). By \Cref{prop:antipodal}(i) and
\eqref{eq:YDX}, the rays through \(\pm p_a\), \(\pm p_b\) are invariant under
\(X_s\), and the abnormal trajectories on them satisfy
\(\dot r=\rho(p)\,r^3\). Following the symmetric component numerically, one
checks that its arcs connect these equilibria to points where \(D=0\).
Numerically, the maximal compact invariant subset of the admissible abnormal
set \(\Sigma_S\cap\set{D\neq0}\) is therefore
\(K_{\max}=\Gamma\cup(-\Gamma)\cup\set{\pm p_a,\pm p_b}\).

\begin{table}[tbp]
\centering
\begin{tabular}{lcccc}
\toprule
equilibrium & coordinates & \(D\) & \(\rho=\kappa/D\) & \(u_s=-D'/D\)\\
\midrule
\(p_a\) & \((0.829003,\,0.082795,\,0.553081)\) & \(-0.04553\) & \(-1.35164\) & \(0.52472\)\\
\(p_b\) & \((0.625817,\,0.538315,\,0.564421)\) & \(-2.47462\) & \(\phantom{-}0.31898\) & \(1.26136\)\\
\bottomrule
\end{tabular}
\caption{Equilibria of \(Y_s\) on \(\Sigma_S\) with \(D\neq0\). The values at
\(-p_a,-p_b\) follow from the parities of \Cref{prop:antipodal}: \(D\) changes
sign, \(\rho\) and \(\abs{u_s}\) do not.}
\label{tab:equilibria}
\end{table}

The component \(\Gamma\) is a small oval around the direction
\((0.36,0.88,-0.30)\), on which \(D>0\). We computed it by integrating \(Y_s\) from
a point of \(\Sigma_S\) (adaptive Runge--Kutta method with relative tolerance
\(10^{-12}\)), detecting the first return. The closure error is
\(1.7\cdot10^{-13}\) and \(\abs{D''}\le2.6\cdot10^{-13}\) along the computed orbit. The
relevant quantities are collected in \Cref{tab:gamma}.

\begin{table}[tbp]
\centering
\begin{tabular}{lll}
\toprule
quantity & value & reference\\
\midrule
\(\min_\Gamma D\), \(\max_\Gamma D\) & \(7.2360\), \(8.3381\) & \Cref{def:regularcomponent}\\
\(\min_\Gamma\abs{Y_s}\), \(\min_\Gamma\abs{\nabla_{\Sph^2}D''}\) & \(0.1993\), \(0.1546\) & \Cref{def:regularcomponent}\\
\(T^\sigma_\Gamma\), \(T^\tau_\Gamma\) & \(1.102812\), \(8.598417\) & \Cref{prop:lift}\\
\(c(\Gamma,Y_s,\abs{D})=T^\tau_\Gamma/T^\sigma_\Gamma\) & \(7.796812\) & \Cref{cor:radial}(a)\\
\(\lambda_\Gamma=\int_\Gamma k\dd\sigma\) & \(0.12406809\) & \Cref{thm:multiplier}\\
\(\int_\Gamma\dvg_{\Sph^2}Y_s\dd\sigma\) & \(0.12406809\) & \Cref{thm:multiplier}\\
multiplier \(e^{\lambda_\Gamma}\); return map, finite differences & \(1.13209\); \(1.1322\) & \Cref{thm:multiplier}\\
\(c(\Gamma,\widehat X_s,\rho)=\ov\rho_\Gamma\) & \(-7.2146\cdot10^{-3}\) & \Cref{cor:radial}(b)\\
\(\osc\chi_\Gamma\) for the cost \(\rho\) & \(0.19789\) & \Cref{prop:lift}(d)\\
\(m_\Gamma=\max_\Gamma\abs{u_s}\); \(m_\Gamma^{-1/3}\); \(m_\Gamma^{-1/3}e^{-\osc\chi_\Gamma}\) & \(0.03618\); \(3.0236\); \(2.4807\) & \Cref{prop:lift}(d)\\
\bottomrule
\end{tabular}
\caption{Numerical data for the regular periodic abnormal component \(\Gamma\).}
\label{tab:gamma}
\end{table}

Since \(\lambda_\Gamma>0\) and \(D>0\) on \(\Gamma\), \(\Gamma\) is a repelling hyperbolic
limit cycle, for both clocks; \(-\Gamma\) is attracting for \(\sigma\) and repelling
for \(\tau\). Inside the oval \(D''<0\), so the normal singular arcs are
elliptic inside and hyperbolic outside, and both families spiral away from
\(\Gamma\) in physical time (\Cref{thm:multiplier}). Since \(\ov\rho_\Gamma<0\), the
abnormal trajectories on the cone over \(\Gamma\) converge to the origin in
infinite time, and those starting at \(\abs{v_0}\le2.48\) satisfy \(\abs{u}\le1\)
for all \(t\ge0\) (\Cref{prop:lift}). As a consistency check, we integrated
\(\dot v=X_s(v)\) in \(\R^3\) from a point of the cone with \(\abs{v_0}=1\): the
trajectory stays on the cone (\(\abs{D''(v/\abs{v})}\le5\cdot10^{-11}\) up to
\(t=400\)) and \(\abs{v(50)}=0.716\), in agreement with
\(\log\abs{v}=\ov\rho_\Gamma\tau+\chi_\Gamma\). \Cref{fig:example}(b) shows the cost
\(\rho\) along \(\Gamma\) and the corresponding corrector, that is, the periodic part
of \(\log\abs{v}\).

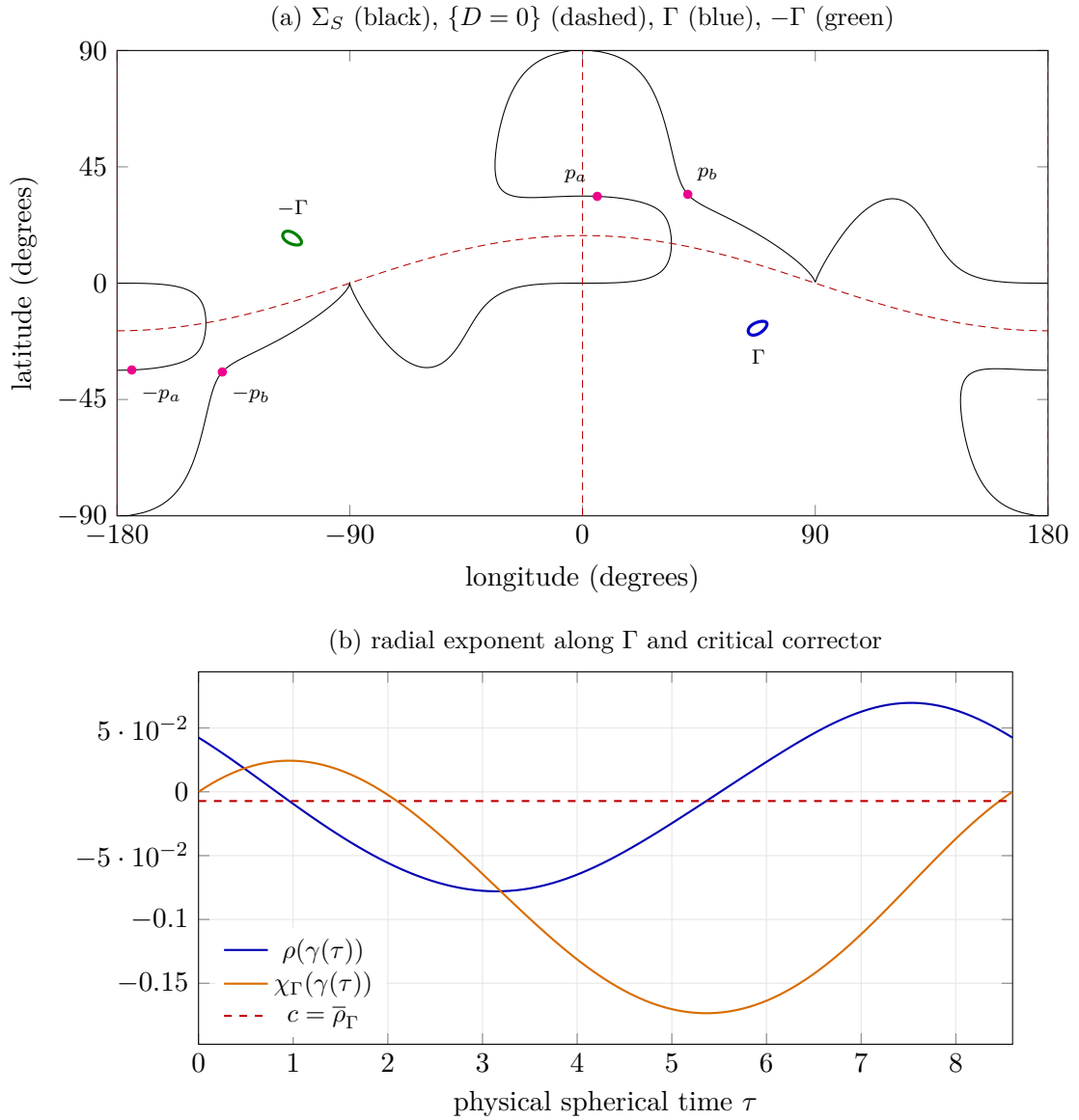
\begin{figure}[p]
\centering
\tikzset{
  sing/.style={black,thin},
  dzero/.style={red!70!black,densely dashed,thin},
  Gamma/.style={blue!80!black,very thick},
  mGamma/.style={green!50!black,very thick},
  rhoplot/.style={blue!70!black,thick},
  chiplot/.style={orange!85!black,thick}
}
\begin{tikzpicture}
\begin{axis}[
  width=0.95\linewidth, height=0.5\linewidth,
  xmin=-180,xmax=180,ymin=-90,ymax=90,
  xtick={-180,-90,0,90,180}, ytick={-90,-45,0,45,90},
  xlabel={longitude (degrees)}, ylabel={latitude (degrees)},
  title={(a) \(\Sigma_S\) (black), \(\set{D=0}\) (dashed), \(\Gamma\) (blue), \(-\Gamma\) (green)},
  axis equal image, every axis title/.style={font=\small,at={(0.5,1.02)},anchor=south}
]
% Sigma_S
\addplot[sing] coordinates {(180.00,-90.00) (178.50,-90.00) (177.00,-90.00) (175.50,-90.00) (174.25,-89.69) (173.00,-89.45) (171.75,-89.20) (170.50,-88.91) (169.25,-88.55) (168.25,-88.23) (167.00,-87.79) (166.00,-87.38) (165.00,-86.93) (164.00,-86.42) (163.00,-85.85) (162.07,-85.25) (161.25,-84.68) (160.36,-84.00) (159.50,-83.27) (158.75,-82.58) (158.00,-81.82) (157.25,-81.00) (156.63,-80.25) (156.00,-79.45) (155.32,-78.50) (154.75,-77.64) (154.20,-76.75) (153.62,-75.75) (153.09,-74.75) (152.59,-73.75) (152.12,-72.75) (151.68,-71.75) (151.25,-70.71) (150.79,-69.50) (150.42,-68.50) (150.00,-67.26) (149.68,-66.25) (149.30,-65.00) (149.00,-63.95) (148.68,-62.75) (148.36,-61.50) (148.07,-60.25) (147.79,-59.00) (147.54,-57.75) (147.30,-56.50) (147.08,-55.25) (146.88,-54.00) (146.70,-52.75) (146.52,-51.25) (146.39,-50.00) (146.27,-48.50) (146.20,-47.25) (146.16,-45.75) (146.18,-44.25) (146.25,-43.00) (146.43,-41.50) (146.67,-40.25) (147.00,-39.11) (147.45,-38.00) (148.00,-37.04) (148.60,-36.25) (149.38,-35.50) (150.25,-34.88) (151.25,-34.36) (152.25,-33.98) (153.50,-33.64) (154.75,-33.41) (156.14,-33.25) (157.50,-33.16) (159.00,-33.11) (160.50,-33.10) (162.00,-33.12) (163.50,-33.16) (165.00,-33.22) (166.25,-33.27) (167.75,-33.34) (169.25,-33.41) (170.75,-33.47) (172.00,-33.52) (173.50,-33.57) (175.00,-33.62) (176.50,-33.65) (178.00,-33.68) (179.50,-33.69)};
\addplot[sing] coordinates {(-180.00,-33.69) (-178.50,-33.68) (-177.00,-33.66) (-175.50,-33.62) (-174.00,-33.57) (-172.64,-33.50) (-171.25,-33.42) (-169.75,-33.31) (-168.50,-33.20) (-167.00,-33.04) (-165.75,-32.90) (-164.50,-32.73) (-163.00,-32.51) (-161.75,-32.30) (-160.50,-32.06) (-159.25,-31.79) (-158.05,-31.50) (-157.00,-31.22) (-155.75,-30.83) (-154.75,-30.48) (-153.55,-30.00) (-152.50,-29.50) (-151.60,-29.00) (-150.75,-28.44) (-149.88,-27.75) (-149.11,-27.00) (-148.49,-26.25) (-147.82,-25.25) (-147.30,-24.25) (-146.88,-23.25) (-146.50,-22.08) (-146.23,-21.00) (-145.98,-19.75) (-145.78,-18.25) (-145.67,-17.00) (-145.63,-15.50) (-145.66,-14.00) (-145.77,-12.75) (-145.99,-11.25) (-146.25,-10.09) (-146.58,-9.00) (-147.00,-7.97) (-147.51,-7.00) (-148.18,-6.00) (-148.81,-5.25) (-149.60,-4.50) (-150.50,-3.81) (-151.40,-3.25) (-152.38,-2.75) (-153.50,-2.28) (-154.50,-1.93) (-155.75,-1.57) (-157.00,-1.28) (-158.25,-1.04) (-159.50,-0.84) (-160.75,-0.67) (-162.25,-0.51) (-163.50,-0.40) (-165.00,-0.29) (-166.25,-0.22) (-167.75,-0.15) (-169.25,-0.10) (-170.75,-0.06) (-172.25,-0.04) (-173.75,-0.02) (-175.25,-0.01) (-176.75,-0.00) (-178.25,-0.00) (-179.75,-0.00)};
\addplot[sing] coordinates {(180.00,0.00) (178.50,0.00) (177.00,0.00) (175.50,0.01) (174.00,0.02) (172.50,0.03) (171.00,0.06) (169.50,0.09) (168.00,0.14) (166.50,0.20) (165.25,0.27) (163.75,0.37) (162.25,0.49) (161.00,0.62) (159.75,0.76) (158.25,0.98) (157.00,1.19) (155.75,1.43) (154.50,1.72) (153.45,2.00) (152.25,2.36) (151.14,2.75) (150.00,3.21) (149.00,3.66) (148.00,4.18) (147.02,4.75) (146.24,5.25) (145.25,5.97) (144.50,6.57) (143.72,7.25) (142.94,8.00) (142.23,8.75) (141.50,9.58) (140.75,10.50) (140.18,11.25) (139.50,12.19) (138.94,13.00) (138.29,14.00) (137.75,14.86) (137.20,15.75) (136.60,16.75) (136.01,17.75) (135.50,18.62) (134.98,19.50) (134.39,20.50) (133.79,21.50) (133.25,22.37) (132.69,23.25) (132.03,24.25) (131.50,25.01) (130.78,26.00) (130.19,26.75) (129.50,27.57) (128.75,28.39) (128.00,29.13) (127.25,29.79) (126.33,30.50) (125.50,31.05) (124.50,31.59) (123.50,32.02) (122.25,32.40) (121.00,32.61) (119.50,32.67) (118.00,32.53) (116.75,32.27) (115.75,31.97) (114.54,31.50) (113.50,31.01) (112.56,30.50) (111.74,30.00) (110.75,29.34) (109.94,28.75) (109.00,28.01) (108.25,27.37) (107.50,26.70) (106.75,25.99) (106.00,25.25) (105.25,24.47) (104.50,23.65) (103.75,22.80) (103.07,22.00) (102.46,21.25) (101.75,20.36) (101.09,19.50) (100.50,18.71) (99.80,17.75) (99.25,16.97) (98.58,16.00) (98.00,15.14) (97.41,14.25) (96.76,13.25) (96.25,12.44) (95.67,11.50) (95.06,10.50) (94.50,9.55) (94.00,8.69) (93.46,7.75) (92.91,6.75) (92.37,5.75) (91.86,4.75) (91.36,3.75) (90.90,2.75) (90.49,1.75) (90.03,0.50) (89.75,0.89) (89.25,1.86) (88.66,2.75) (88.00,3.57) (87.25,4.43) (86.50,5.22) (85.75,5.96) (85.00,6.66) (84.25,7.34) (83.49,8.00) (82.60,8.75) (81.75,9.44) (81.00,10.04) (80.08,10.75) (79.25,11.38) (78.41,12.00) (77.50,12.66) (76.66,13.25) (75.75,13.88) (74.85,14.50) (74.00,15.07) (73.00,15.73) (72.19,16.25) (71.25,16.85) (70.25,17.47) (69.39,18.00) (68.50,18.53) (67.50,19.13) (66.50,19.71) (65.56,20.25) (64.67,20.75) (63.75,21.26) (62.75,21.81) (61.75,22.35) (60.75,22.88) (59.75,23.40) (58.75,23.92) (57.75,24.43) (56.75,24.93) (55.75,25.43) (54.75,25.92) (53.75,26.41) (52.75,26.90) (51.75,27.39) (50.75,27.89) (49.75,28.38) (48.75,28.89) (47.75,29.40) (46.75,29.94) (45.75,30.50) (44.90,31.00) (44.00,31.58) (43.04,32.25) (42.25,32.87) (41.50,33.54) (40.75,34.31) (40.00,35.21) (39.45,36.00) (38.85,37.00) (38.34,38.00) (37.90,39.00) (37.50,40.04) (37.08,41.25) (36.75,42.32) (36.41,43.50) (36.07,44.75) (35.75,46.00) (35.50,47.03) (35.21,48.25) (34.91,49.50) (34.62,50.75) (34.32,52.00) (34.03,53.25) (33.75,54.43) (33.49,55.50) (33.18,56.75) (32.87,58.00) (32.54,59.25) (32.25,60.34) (31.93,61.50) (31.57,62.75) (31.25,63.82) (30.88,65.00) (30.50,66.18) (30.13,67.25) (29.75,68.32) (29.31,69.50) (28.91,70.50) (28.49,71.50) (28.00,72.60) (27.50,73.66) (27.00,74.65) (26.50,75.58) (25.97,76.50) (25.36,77.50) (24.75,78.42) (24.16,79.25) (23.50,80.11) (22.75,81.00) (22.07,81.75) (21.33,82.50) (20.53,83.25) (19.75,83.91) (18.99,84.50) (18.00,85.21) (17.15,85.75) (16.25,86.27) (15.25,86.80) (14.25,87.27) (13.08,87.75) (12.00,88.16) (10.89,88.50) (9.75,88.84) (8.50,89.15) (7.25,89.41) (6.00,89.65) (4.60,89.75) (3.25,90.00) (1.75,90.00) (0.25,90.00)};
\addplot[sing] coordinates {(-0.00,90.00) (-1.50,90.00) (-3.00,90.00) (-4.50,90.00) (-5.75,89.69) (-7.00,89.45) (-8.25,89.20) (-9.50,88.91) (-10.75,88.55) (-11.75,88.23) (-13.00,87.79) (-14.00,87.38) (-15.00,86.93) (-16.00,86.42) (-17.00,85.85) (-17.93,85.25) (-18.75,84.68) (-19.64,84.00) (-20.50,83.27) (-21.25,82.58) (-22.00,81.82) (-22.75,81.00) (-23.37,80.25) (-24.00,79.45) (-24.68,78.50) (-25.25,77.64) (-25.80,76.75) (-26.38,75.75) (-26.91,74.75) (-27.41,73.75) (-27.88,72.75) (-28.32,71.75) (-28.75,70.71) (-29.21,69.50) (-29.58,68.50) (-30.00,67.26) (-30.32,66.25) (-30.70,65.00) (-31.00,63.95) (-31.32,62.75) (-31.64,61.50) (-31.93,60.25) (-32.21,59.00) (-32.46,57.75) (-32.70,56.50) (-32.92,55.25) (-33.12,54.00) (-33.30,52.75) (-33.48,51.25) (-33.61,50.00) (-33.73,48.50) (-33.80,47.25) (-33.84,45.75) (-33.82,44.25) (-33.75,43.00) (-33.57,41.50) (-33.33,40.25) (-33.00,39.11) (-32.55,38.00) (-32.00,37.04) (-31.40,36.25) (-30.62,35.50) (-29.75,34.88) (-28.75,34.36) (-27.75,33.98) (-26.50,33.64) (-25.25,33.41) (-23.86,33.25) (-22.50,33.16) (-21.00,33.11) (-19.50,33.10) (-18.00,33.12) (-16.50,33.16) (-15.00,33.22) (-13.75,33.27) (-12.25,33.34) (-10.75,33.41) (-9.25,33.47) (-8.00,33.52) (-6.50,33.57) (-5.00,33.62) (-3.50,33.65) (-2.00,33.68) (-0.50,33.69) (1.00,33.69) (2.50,33.67) (4.00,33.64) (5.50,33.59) (7.00,33.52) (8.25,33.45) (9.75,33.34) (11.00,33.24) (12.50,33.10) (13.75,32.96) (15.25,32.77) (16.50,32.59) (17.75,32.38) (19.00,32.16) (20.25,31.90) (21.50,31.61) (22.75,31.29) (23.75,30.99) (25.00,30.57) (26.00,30.19) (27.00,29.75) (28.00,29.23) (29.00,28.62) (29.83,28.00) (30.65,27.25) (31.32,26.50) (32.00,25.55) (32.50,24.67) (33.00,23.56) (33.37,22.50) (33.72,21.25) (33.98,20.00) (34.16,18.75) (34.29,17.50) (34.37,16.00) (34.36,14.50) (34.26,13.00) (34.10,11.75) (33.85,10.50) (33.50,9.25) (33.12,8.25) (32.64,7.25) (32.01,6.25) (31.41,5.50) (30.68,4.75) (29.77,4.00) (29.00,3.49) (28.00,2.93) (27.00,2.48) (25.75,2.02) (24.75,1.71) (23.50,1.39) (22.25,1.13) (21.00,0.91) (19.75,0.73) (18.25,0.56) (17.00,0.44) (15.50,0.32) (14.25,0.24) (12.75,0.17) (11.25,0.12) (9.75,0.07) (8.25,0.04) (6.75,0.02) (5.25,0.01) (3.75,0.00) (2.25,0.00) (0.75,0.00) (-0.50,-0.00) (-2.00,-0.00) (-3.50,-0.00) (-5.00,-0.01) (-6.50,-0.02) (-8.00,-0.04) (-9.50,-0.07) (-11.00,-0.11) (-12.50,-0.16) (-14.00,-0.23) (-15.25,-0.30) (-16.75,-0.41) (-18.00,-0.51) (-19.50,-0.67) (-20.75,-0.83) (-22.00,-1.02) (-23.34,-1.25) (-24.56,-1.50) (-25.75,-1.78) (-27.00,-2.13) (-28.16,-2.50) (-29.25,-2.90) (-30.25,-3.32) (-31.25,-3.79) (-32.25,-4.32) (-33.25,-4.92) (-34.12,-5.50) (-35.00,-6.16) (-35.75,-6.78) (-36.55,-7.50) (-37.30,-8.25) (-38.00,-9.00) (-38.75,-9.88) (-39.45,-10.75) (-40.01,-11.50) (-40.72,-12.50) (-41.25,-13.29) (-41.87,-14.25) (-42.49,-15.25) (-43.00,-16.09) (-43.55,-17.00) (-44.14,-18.00) (-44.72,-19.00) (-45.25,-19.89) (-45.76,-20.75) (-46.37,-21.75) (-46.99,-22.75) (-47.50,-23.54) (-48.14,-24.50) (-48.75,-25.36) (-49.41,-26.25) (-50.01,-27.00) (-50.75,-27.86) (-51.50,-28.65) (-52.25,-29.36) (-53.00,-30.00) (-54.00,-30.73) (-54.85,-31.25) (-55.84,-31.75) (-57.00,-32.19) (-58.24,-32.50) (-59.50,-32.66) (-61.00,-32.65) (-62.25,-32.49) (-63.50,-32.20) (-64.75,-31.79) (-65.75,-31.37) (-66.75,-30.88) (-67.75,-30.32) (-68.65,-29.75) (-69.50,-29.16) (-70.38,-28.50) (-71.25,-27.80) (-72.00,-27.15) (-72.75,-26.47) (-73.50,-25.75) (-74.25,-24.99) (-75.00,-24.20) (-75.75,-23.37) (-76.50,-22.51) (-77.13,-21.75) (-77.75,-20.99) (-78.50,-20.04) (-79.10,-19.25) (-79.75,-18.37) (-80.38,-17.50) (-81.00,-16.61) (-81.59,-15.75) (-82.25,-14.77) (-82.75,-14.00) (-83.40,-13.00) (-84.00,-12.04) (-84.50,-11.22) (-85.09,-10.25) (-85.68,-9.25) (-86.25,-8.25) (-86.75,-7.36) (-87.25,-6.45) (-87.76,-5.50) (-88.27,-4.50) (-88.75,-3.50) (-89.25,-2.38) (-89.71,-1.25) (-90.00,0.00) (-90.39,-1.25) (-91.00,-2.25) (-91.53,-3.00) (-92.25,-3.87) (-93.00,-4.70) (-93.75,-5.47) (-94.50,-6.20) (-95.25,-6.89) (-96.00,-7.56) (-96.80,-8.25) (-97.70,-9.00) (-98.50,-9.64) (-99.27,-10.25) (-100.25,-11.00) (-101.00,-11.56) (-101.93,-12.25) (-102.75,-12.83) (-103.69,-13.50) (-104.50,-14.06) (-105.50,-14.73) (-106.27,-15.25) (-107.25,-15.89) (-108.20,-16.50) (-109.00,-17.00) (-110.00,-17.62) (-111.00,-18.23) (-111.86,-18.75) (-112.75,-19.27) (-113.75,-19.85) (-114.75,-20.42) (-115.75,-20.99) (-116.68,-21.50) (-117.60,-22.00) (-118.53,-22.50) (-119.50,-23.01) (-120.50,-23.53) (-121.50,-24.05) (-122.50,-24.55) (-123.50,-25.06) (-124.50,-25.55) (-125.50,-26.05) (-126.50,-26.54) (-127.50,-27.03) (-128.50,-27.52) (-129.50,-28.01) (-130.50,-28.51) (-131.50,-29.01) (-132.50,-29.53) (-133.50,-30.07) (-134.50,-30.64) (-135.50,-31.25) (-136.26,-31.75) (-137.25,-32.47) (-138.00,-33.09) (-138.75,-33.78) (-139.50,-34.59) (-140.21,-35.50) (-140.75,-36.31) (-141.28,-37.25) (-141.77,-38.25) (-142.25,-39.38) (-142.66,-40.50) (-143.00,-41.51) (-143.38,-42.75) (-143.73,-44.00) (-144.00,-45.03) (-144.31,-46.25) (-144.61,-47.50) (-144.91,-48.75) (-145.21,-50.00) (-145.50,-51.25) (-145.75,-52.32) (-146.03,-53.50) (-146.33,-54.75) (-146.63,-56.00) (-146.94,-57.25) (-147.25,-58.45) (-147.52,-59.50) (-147.86,-60.75) (-148.21,-62.00) (-148.50,-63.00) (-148.88,-64.25) (-149.25,-65.41) (-149.61,-66.50) (-150.00,-67.63) (-150.41,-68.75) (-150.79,-69.75) (-151.25,-70.88) (-151.73,-72.00) (-152.19,-73.00) (-152.67,-74.00) (-153.19,-75.00) (-153.74,-76.00) (-154.25,-76.87) (-154.80,-77.75) (-155.48,-78.75) (-156.03,-79.50) (-156.75,-80.41) (-157.47,-81.25) (-158.17,-82.00) (-158.93,-82.75) (-159.75,-83.49) (-160.50,-84.12) (-161.35,-84.75) (-162.25,-85.38) (-163.25,-85.99) (-164.17,-86.50) (-165.17,-87.00) (-166.25,-87.48) (-167.25,-87.89) (-168.32,-88.25) (-169.50,-88.64) (-170.75,-88.96) (-172.00,-89.23) (-173.25,-89.48) (-174.50,-89.71) (-175.75,-90.00) (-177.25,-90.00) (-178.75,-90.00)};
\addplot[sing] coordinates {(65.25,-20.01) (64.45,-19.50) (64.09,-18.25) (64.50,-17.19) (65.25,-16.36) (66.13,-15.75) (67.20,-15.25) (68.25,-14.93) (69.75,-14.76) (71.00,-15.12) (71.27,-16.25) (70.77,-17.25) (70.16,-18.00) (69.34,-18.75) (68.50,-19.33) (67.50,-19.82) (66.25,-20.11)};
\addplot[sing] coordinates {(-112.00,14.99) (-110.50,14.76) (-109.17,15.00) (-108.68,16.00) (-109.07,17.00) (-109.75,17.90) (-110.50,18.61) (-111.37,19.25) (-112.33,19.75) (-113.50,20.08) (-114.77,20.00) (-115.73,19.25) (-115.87,18.00) (-115.37,17.00) (-114.62,16.25) (-113.75,15.68) (-112.75,15.23)};
\addplot[Gamma] coordinates {(64.41,-17.33) (64.36,-17.41) (64.32,-17.49) (64.28,-17.56) (64.25,-17.64) (64.22,-17.72) (64.19,-17.80) (64.17,-17.87) (64.14,-17.95) (64.12,-18.02) (64.11,-18.10) (64.09,-18.17) (64.08,-18.25) (64.08,-18.32) (64.07,-18.39) (64.07,-18.46) (64.07,-18.53) (64.08,-18.60) (64.08,-18.67) (64.09,-18.74) (64.11,-18.81) (64.12,-18.87) (64.14,-18.93) (64.16,-19.00) (64.18,-19.06) (64.21,-19.12) (64.23,-19.18) (64.26,-19.24) (64.30,-19.29) (64.33,-19.35) (64.37,-19.40) (64.41,-19.45) (64.45,-19.50) (64.50,-19.55) (64.54,-19.60) (64.59,-19.64) (64.64,-19.69) (64.70,-19.73) (64.75,-19.77) (64.81,-19.80) (64.87,-19.84) (64.93,-19.87) (64.99,-19.91) (65.06,-19.94) (65.13,-19.96) (65.20,-19.99) (65.27,-20.01) (65.34,-20.03) (65.42,-20.05) (65.50,-20.07) (65.58,-20.08) (65.66,-20.10) (65.74,-20.11) (65.82,-20.11) (65.91,-20.12) (66.00,-20.12) (66.09,-20.12) (66.18,-20.12) (66.27,-20.11) (66.37,-20.10) (66.46,-20.09) (66.56,-20.08) (66.66,-20.06) (66.76,-20.04) (66.86,-20.02) (66.96,-19.99) (67.07,-19.96) (67.17,-19.93) (67.28,-19.90) (67.38,-19.86) (67.49,-19.82) (67.60,-19.78) (67.71,-19.74) (67.82,-19.69) (67.93,-19.64) (68.04,-19.58) (68.15,-19.53) (68.26,-19.47) (68.37,-19.41) (68.48,-19.34) (68.59,-19.28) (68.70,-19.21) (68.81,-19.14) (68.92,-19.06) (69.02,-18.99) (69.13,-18.91) (69.24,-18.83) (69.34,-18.75) (69.45,-18.66) (69.55,-18.58) (69.65,-18.49) (69.75,-18.40) (69.84,-18.31) (69.94,-18.22) (70.03,-18.13) (70.12,-18.04) (70.21,-17.95) (70.29,-17.85) (70.38,-17.76) (70.46,-17.66) (70.53,-17.57) (70.61,-17.47) (70.68,-17.38) (70.75,-17.28) (70.81,-17.19) (70.87,-17.09) (70.93,-17.00) (70.99,-16.91) (71.04,-16.82) (71.08,-16.73) (71.12,-16.64) (71.16,-16.55) (71.20,-16.46) (71.23,-16.38) (71.26,-16.29) (71.28,-16.21) (71.30,-16.13) (71.31,-16.05) (71.32,-15.98) (71.33,-15.90) (71.33,-15.83) (71.33,-15.76) (71.32,-15.69) (71.31,-15.62) (71.30,-15.56) (71.28,-15.50) (71.25,-15.44) (71.23,-15.38) (71.19,-15.32) (71.16,-15.27) (71.12,-15.22) (71.07,-15.17) (71.02,-15.13) (70.97,-15.09) (70.92,-15.05) (70.85,-15.01) (70.79,-14.97) (70.72,-14.94) (70.65,-14.91) (70.58,-14.88) (70.50,-14.86) (70.41,-14.84) (70.33,-14.82) (70.24,-14.80) (70.15,-14.79) (70.05,-14.78) (69.95,-14.77) (69.85,-14.76) (69.75,-14.76) (69.64,-14.76) (69.53,-14.76) (69.42,-14.76) (69.31,-14.77) (69.19,-14.78) (69.08,-14.79) (68.96,-14.80) (68.84,-14.82) (68.72,-14.84) (68.59,-14.86) (68.47,-14.88) (68.35,-14.91) (68.22,-14.94) (68.09,-14.97) (67.97,-15.00) (67.84,-15.04) (67.72,-15.07) (67.59,-15.11) (67.46,-15.15) (67.34,-15.20) (67.22,-15.24) (67.09,-15.29) (66.97,-15.34) (66.85,-15.39) (66.73,-15.44) (66.61,-15.50) (66.49,-15.55) (66.38,-15.61) (66.27,-15.67) (66.16,-15.73) (66.05,-15.79) (65.94,-15.86) (65.84,-15.92) (65.74,-15.99) (65.64,-16.06) (65.54,-16.13) (65.45,-16.20) (65.36,-16.27) (65.27,-16.34) (65.18,-16.41) (65.10,-16.49) (65.02,-16.56) (64.95,-16.64) (64.87,-16.71) (64.80,-16.79) (64.74,-16.87) (64.67,-16.94) (64.61,-17.02) (64.56,-17.10) (64.50,-17.17) (64.45,-17.25) (64.41,-17.33)};
\addplot[mGamma] coordinates {(-115.59,17.33) (-115.64,17.41) (-115.68,17.49) (-115.72,17.56) (-115.75,17.64) (-115.78,17.72) (-115.81,17.80) (-115.83,17.87) (-115.86,17.95) (-115.88,18.02) (-115.89,18.10) (-115.91,18.17) (-115.92,18.25) (-115.92,18.32) (-115.93,18.39) (-115.93,18.46) (-115.93,18.53) (-115.92,18.60) (-115.92,18.67) (-115.91,18.74) (-115.89,18.81) (-115.88,18.87) (-115.86,18.93) (-115.84,19.00) (-115.82,19.06) (-115.79,19.12) (-115.77,19.18) (-115.74,19.24) (-115.70,19.29) (-115.67,19.35) (-115.63,19.40) (-115.59,19.45) (-115.55,19.50) (-115.50,19.55) (-115.46,19.60) (-115.41,19.64) (-115.36,19.69) (-115.30,19.73) (-115.25,19.77) (-115.19,19.80) (-115.13,19.84) (-115.07,19.87) (-115.01,19.91) (-114.94,19.94) (-114.87,19.96) (-114.80,19.99) (-114.73,20.01) (-114.66,20.03) (-114.58,20.05) (-114.50,20.07) (-114.42,20.08) (-114.34,20.10) (-114.26,20.11) (-114.18,20.11) (-114.09,20.12) (-114.00,20.12) (-113.91,20.12) (-113.82,20.12) (-113.73,20.11) (-113.63,20.10) (-113.54,20.09) (-113.44,20.08) (-113.34,20.06) (-113.24,20.04) (-113.14,20.02) (-113.04,19.99) (-112.93,19.96) (-112.83,19.93) (-112.72,19.90) (-112.62,19.86) (-112.51,19.82) (-112.40,19.78) (-112.29,19.74) (-112.18,19.69) (-112.07,19.64) (-111.96,19.58) (-111.85,19.53) (-111.74,19.47) (-111.63,19.41) (-111.52,19.34) (-111.41,19.28) (-111.30,19.21) (-111.19,19.14) (-111.08,19.06) (-110.98,18.99) (-110.87,18.91) (-110.76,18.83) (-110.66,18.75) (-110.55,18.66) (-110.45,18.58) (-110.35,18.49) (-110.25,18.40) (-110.16,18.31) (-110.06,18.22) (-109.97,18.13) (-109.88,18.04) (-109.79,17.95) (-109.71,17.85) (-109.62,17.76) (-109.54,17.66) (-109.47,17.57) (-109.39,17.47) (-109.32,17.38) (-109.25,17.28) (-109.19,17.19) (-109.13,17.09) (-109.07,17.00) (-109.01,16.91) (-108.96,16.82) (-108.92,16.73) (-108.88,16.64) (-108.84,16.55) (-108.80,16.46) (-108.77,16.38) (-108.74,16.29) (-108.72,16.21) (-108.70,16.13) (-108.69,16.05) (-108.68,15.98) (-108.67,15.90) (-108.67,15.83) (-108.67,15.76) (-108.68,15.69) (-108.69,15.62) (-108.70,15.56) (-108.72,15.50) (-108.75,15.44) (-108.77,15.38) (-108.81,15.32) (-108.84,15.27) (-108.88,15.22) (-108.93,15.17) (-108.98,15.13) (-109.03,15.09) (-109.08,15.05) (-109.15,15.01) (-109.21,14.97) (-109.28,14.94) (-109.35,14.91) (-109.42,14.88) (-109.50,14.86) (-109.59,14.84) (-109.67,14.82) (-109.76,14.80) (-109.85,14.79) (-109.95,14.78) (-110.05,14.77) (-110.15,14.76) (-110.25,14.76) (-110.36,14.76) (-110.47,14.76) (-110.58,14.76) (-110.69,14.77) (-110.81,14.78) (-110.92,14.79) (-111.04,14.80) (-111.16,14.82) (-111.28,14.84) (-111.41,14.86) (-111.53,14.88) (-111.65,14.91) (-111.78,14.94) (-111.91,14.97) (-112.03,15.00) (-112.16,15.04) (-112.28,15.07) (-112.41,15.11) (-112.54,15.15) (-112.66,15.20) (-112.78,15.24) (-112.91,15.29) (-113.03,15.34) (-113.15,15.39) (-113.27,15.44) (-113.39,15.50) (-113.51,15.55) (-113.62,15.61) (-113.73,15.67) (-113.84,15.73) (-113.95,15.79) (-114.06,15.86) (-114.16,15.92) (-114.26,15.99) (-114.36,16.06) (-114.46,16.13) (-114.55,16.20) (-114.64,16.27) (-114.73,16.34) (-114.82,16.41) (-114.90,16.49) (-114.98,16.56) (-115.05,16.64) (-115.13,16.71) (-115.20,16.79) (-115.26,16.87) (-115.33,16.94) (-115.39,17.02) (-115.44,17.10) (-115.50,17.17) (-115.55,17.25) (-115.59,17.33)};
% D=0
\addplot[dzero] coordinates {(-180.0,-18.43) (-175.0,-18.37) (-170.0,-18.17) (-165.0,-17.85) (-160.0,-17.39) (-155.0,-16.81) (-150.0,-16.10) (-145.0,-15.27) (-140.0,-14.32) (-135.0,-13.26) (-130.0,-12.09) (-125.0,-10.82) (-120.0,-9.46) (-115.0,-8.02) (-110.0,-6.50) (-105.0,-4.93) (-100.0,-3.31) (-95.0,-1.66) (-90.0,0.00) (-85.0,1.66) (-80.0,3.31) (-75.0,4.93) (-70.0,6.50) (-65.0,8.02) (-60.0,9.46) (-55.0,10.82) (-50.0,12.09) (-45.0,13.26) (-40.0,14.32) (-35.0,15.27) (-30.0,16.10) (-25.0,16.81) (-20.0,17.39) (-15.0,17.85) (-10.0,18.17) (-5.0,18.37) (0.0,18.43) (5.0,18.37) (10.0,18.17) (15.0,17.85) (20.0,17.39) (25.0,16.81) (30.0,16.10) (35.0,15.27) (40.0,14.32) (45.0,13.26) (50.0,12.09) (55.0,10.82) (60.0,9.46) (65.0,8.02) (70.0,6.50) (75.0,4.93) (80.0,3.31) (85.0,1.66) (90.0,0.00) (95.0,-1.66) (100.0,-3.31) (105.0,-4.93) (110.0,-6.50) (115.0,-8.02) (120.0,-9.46) (125.0,-10.82) (130.0,-12.09) (135.0,-13.26) (140.0,-14.32) (145.0,-15.27) (150.0,-16.10) (155.0,-16.81) (160.0,-17.39) (165.0,-17.85) (170.0,-18.17) (175.0,-18.37) (180.0,-18.43)};
\addplot[dzero] coordinates {(0,-90) (0,90)};
\addplot[dzero] coordinates {(180,-90) (180,90)};
\addplot[dzero] coordinates {(-180,-90) (-180,90)};
\addplot[only marks,mark=*,mark size=1.6pt,magenta] coordinates {(5.70,33.58) (-174.30,-33.58) (40.70,34.36) (-139.30,-34.36)};
\node[font=\scriptsize,anchor=south east] at (axis cs:5.70,35) {\(p_a\)};
\node[font=\scriptsize,anchor=south west] at (axis cs:40.70,36) {\(p_b\)};
\node[font=\scriptsize,anchor=north west] at (axis cs:-174.30,-36) {\(-p_a\)};
\node[font=\scriptsize,anchor=north west] at (axis cs:-139.30,-36) {\(-p_b\)};
\node[font=\scriptsize,anchor=north] at (axis cs:68,-22) {\(\Gamma\)};
\node[font=\scriptsize,anchor=south] at (axis cs:-112,22) {\(-\Gamma\)};
\end{axis}
\end{tikzpicture}

\vspace{2mm}
\begin{tikzpicture}
\begin{axis}[
  width=0.8\linewidth, height=0.42\linewidth,
  xmin=0,xmax=8.6,
  xlabel={physical spherical time \(\tau\)},
  legend style={at={(0.02,0.02)},anchor=south west,draw=none,fill=white,font=\small},
  title={(b) radial exponent along \(\Gamma\) and critical corrector},
  every axis title/.style={font=\small,at={(0.5,1.02)},anchor=south},
  grid=major, grid style={gray!20}
]
\addplot[rhoplot] coordinates {(0.0000,0.04250) (0.0710,0.03921) (0.1426,0.03578) (0.2145,0.03224) (0.2861,0.02862) (0.3583,0.02490) (0.4308,0.02110) (0.5030,0.01728) (0.5759,0.01338) (0.6489,0.00945) (0.7218,0.00551) (0.7952,0.00155) (0.8689,-0.00242) (0.9423,-0.00635) (1.0163,-0.01029) (1.0905,-0.01419) (1.1645,-0.01804) (1.2390,-0.02186) (1.3138,-0.02561) (1.3882,-0.02928) (1.4633,-0.03290) (1.5385,-0.03643) (1.6135,-0.03985) (1.6890,-0.04319) (1.7647,-0.04642) (1.8400,-0.04952) (1.9160,-0.05251) (1.9920,-0.05537) (2.0677,-0.05808) (2.1440,-0.06066) (2.2204,-0.06308) (2.2964,-0.06533) (2.3729,-0.06743) (2.4495,-0.06936) (2.5257,-0.07109) (2.6025,-0.07265) (2.6792,-0.07402) (2.7555,-0.07518) (2.8323,-0.07616) (2.9091,-0.07692) (2.9854,-0.07747) (3.0621,-0.07782) (3.1388,-0.07795) (3.2149,-0.07786) (3.2914,-0.07756) (3.3679,-0.07705) (3.4437,-0.07632) (3.5199,-0.07538) (3.5959,-0.07424) (3.6714,-0.07290) (3.7470,-0.07137) (3.8225,-0.06964) (3.8973,-0.06775) (3.9723,-0.06568) (4.0471,-0.06345) (4.1211,-0.06108) (4.1953,-0.05857) (4.2692,-0.05593) (4.3424,-0.05319) (4.4157,-0.05033) (4.4887,-0.04739) (4.5609,-0.04438) (4.6332,-0.04129) (4.7051,-0.03813) (4.7763,-0.03495) (4.8476,-0.03171) (4.9186,-0.02844) (4.9888,-0.02516) (5.0591,-0.02185) (5.1290,-0.01852) (5.1983,-0.01521) (5.2676,-0.01187) (5.3367,-0.00854) (5.4051,-0.00522) (5.4736,-0.00190) (5.5419,0.00142) (5.6095,0.00470) (5.6773,0.00799) (5.7450,0.01126) (5.8120,0.01449) (5.8793,0.01772) (5.9465,0.02093) (6.0131,0.02408) (6.0800,0.02722) (6.1468,0.03032) (6.2132,0.03335) (6.2799,0.03635) (6.3465,0.03929) (6.4127,0.04214) (6.4794,0.04493) (6.5460,0.04764) (6.6123,0.05023) (6.6790,0.05272) (6.7458,0.05509) (6.8122,0.05731) (6.8791,0.05940) (6.9461,0.06134) (7.0128,0.06308) (7.0801,0.06466) (7.1474,0.06604) (7.2145,0.06721) (7.2821,0.06817) (7.3499,0.06891) (7.4174,0.06941) (7.4855,0.06967) (7.5537,0.06970) (7.6218,0.06948) (7.6904,0.06901) (7.7592,0.06830) (7.8277,0.06735) (7.8969,0.06615) (7.9663,0.06471) (8.0354,0.06305) (8.1052,0.06116) (8.1751,0.05905) (8.2449,0.05675) (8.3153,0.05425) (8.3858,0.05155) (8.4562,0.04870) (8.5272,0.04568) (8.5984,0.04250)};
\addplot[chiplot] coordinates {(0.0000,0.00000) (0.0710,0.00341) (0.1426,0.00662) (0.2145,0.00958) (0.2861,0.01227) (0.3583,0.01473) (0.4308,0.01692) (0.5030,0.01883) (0.5759,0.02047) (0.6489,0.02183) (0.7218,0.02290) (0.7952,0.02369) (0.8689,0.02419) (0.9423,0.02440) (1.0163,0.02431) (1.0905,0.02394) (1.1645,0.02328) (1.2390,0.02233) (1.3138,0.02110) (1.3882,0.01959) (1.4633,0.01780) (1.5385,0.01573) (1.6135,0.01341) (1.6890,0.01082) (1.7647,0.00797) (1.8400,0.00490) (1.9160,0.00158) (1.9920,-0.00198) (2.0677,-0.00573) (2.1440,-0.00971) (2.2204,-0.01388) (2.2964,-0.01822) (2.3729,-0.02275) (2.4495,-0.02743) (2.5257,-0.03224) (2.6025,-0.03720) (2.6792,-0.04227) (2.7555,-0.04742) (2.8323,-0.05267) (2.9091,-0.05800) (2.9854,-0.06334) (3.0621,-0.06874) (3.1388,-0.07416) (3.2149,-0.07955) (3.2914,-0.08494) (3.3679,-0.09030) (3.4437,-0.09557) (3.5199,-0.10080) (3.5959,-0.10595) (3.6714,-0.11095) (3.7470,-0.11586) (3.8225,-0.12064) (3.8973,-0.12524) (3.9723,-0.12971) (4.0471,-0.13399) (4.1211,-0.13807) (4.1953,-0.14198) (4.2692,-0.14568) (4.3424,-0.14914) (4.4157,-0.15241) (4.4887,-0.15545) (4.5609,-0.15824) (4.6332,-0.16081) (4.7051,-0.16315) (4.7763,-0.16524) (4.8476,-0.16710) (4.9186,-0.16873) (4.9888,-0.17010) (5.0591,-0.17125) (5.1290,-0.17215) (5.1983,-0.17282) (5.2676,-0.17326) (5.3367,-0.17347) (5.4051,-0.17345) (5.4736,-0.17320) (5.5419,-0.17272) (5.6095,-0.17202) (5.6773,-0.17111) (5.7450,-0.16997) (5.8120,-0.16862) (5.8793,-0.16705) (5.9465,-0.16527) (6.0131,-0.16329) (6.0800,-0.16109) (6.1468,-0.15868) (6.2132,-0.15609) (6.2799,-0.15329) (6.3465,-0.15028) (6.4127,-0.14711) (6.4794,-0.14373) (6.5460,-0.14016) (6.6123,-0.13644) (6.6790,-0.13252) (6.7458,-0.12844) (6.8122,-0.12423) (6.8791,-0.11984) (6.9461,-0.11531) (7.0128,-0.11068) (7.0801,-0.10590) (7.1474,-0.10101) (7.2145,-0.09605) (7.2821,-0.09099) (7.3499,-0.08585) (7.4174,-0.08070) (7.4855,-0.07547) (7.5537,-0.07022) (7.6218,-0.06499) (7.6904,-0.05975) (7.7592,-0.05452) (7.8277,-0.04938) (7.8969,-0.04426) (7.9663,-0.03922) (8.0354,-0.03430) (8.1052,-0.02947) (8.1751,-0.02475) (8.2449,-0.02021) (8.3153,-0.01580) (8.3858,-0.01155) (8.4562,-0.00752) (8.5272,-0.00365) (8.5984,-0.00000)};
\addplot[red!75!black,dashed,thick] coordinates {(0,-0.0072146) (8.6,-0.0072146)};
\legend{\(\rho(\gamma(\tau))\),\(\chi_\Gamma(\gamma(\tau))\),\(c=\ov\rho_\Gamma\)}
\end{axis}
\end{tikzpicture}
\caption{The example \eqref{eq:exampleH}. (a) Equirectangular chart of the
sphere. (b) The cost \(\rho\) along \(\Gamma\), its average, and the corrector
\(\chi_\Gamma=\int(\rho-\ov\rho_\Gamma)\dd\tau\), which is the periodic part of
\(\log\abs{v}\) along the abnormal cone.}
\label{fig:example}
\end{figure}

\Cref{tab:selection} applies \Cref{thm:selection} to several compact invariant
sets. For the intrinsic cost \(\rho\) (spherical clock) the periodic component
is selected against the expanding ray through \(\pm p_b\), but not against the
strongly contracting ray through \(\pm p_a\). For the clock density \(\abs{D}\) on
the whole quintic, the minimum is \(0\), attained at the equilibria where \(D=0\),
which illustrates \Cref{rem:intrinsic}.

\begin{table}[tbp]
\centering
\begin{tabular}{llll}
\toprule
\(K\) & clock, cost & critical value \(c\) & \(\wt\cA_s=\wt\cM_s\)\\
\midrule
\(\pm\Gamma\) & \(\widehat X_s\), \(\rho\) & \(-7.2146\cdot10^{-3}\) & \(\pm\Gamma\)\\
\(\pm\Gamma\cup\set{\pm p_b}\) & \(\widehat X_s\), \(\rho\) & \(-7.2146\cdot10^{-3}\) & \(\pm\Gamma\)\\
\(K_{\max}=\pm\Gamma\cup\set{\pm p_a,\pm p_b}\) & \(\widehat X_s\), \(\rho\) & \(-1.35164\) & \(\set{\pm p_a}\)\\
\(\pm\Gamma\) & \(Y_s\), \(\abs{D}\) & \(7.796812\) & \(\pm\Gamma\)\\
\(K_{\max}\) & \(Y_s\), \(\abs{D}\) & \(0.04553\) & \(\set{\pm p_a}\)\\
\(\Sigma_S\) & \(Y_s\), \(\abs{D}\) & \(0\) & \(\wt\cM_s=\set{\pm e_1,\pm e_2,\pm e_3}\)\\
\bottomrule
\end{tabular}
\caption{Singular Aubry--Mather structures for the example. In the last line
\(K=\Sigma_S\) contains non-closed orbits, so \Cref{thm:selection} does not
apply; the critical value and the Mather set are obtained from the ergodic
measures, which are supported on equilibria and on \(\pm\Gamma\), and only the
Mather set is reported.}
\label{tab:selection}
\end{table}

\begin{remark}[Reliability]
The computations are floating-point computations and do not constitute a
proof. The qualitative conclusions (existence of \(\Gamma\), \(D\neq0\) and
\(Y_s\neq0\) on \(\Gamma\), regularity, signs of \(\lambda_\Gamma\) and \(\ov\rho_\Gamma\))
are, however, robust: the corresponding quantities are bounded away from
zero by margins much larger than the numerical errors, and they could be
certified by interval arithmetic. Once certified, they would show,
combined with \Cref{prop:persistence}, that regular periodic abnormal
components exist on an open set of cubic fields.
\end{remark}

% =================================================================
\section{Perspectives}\label{sec:perspectives}
% =================================================================

The singular Aubry--Mather structure of \Cref{sec:AMstructures} is the
minimal variational structure compatible with the rigidity of singular
arcs: the only competitors of an abnormal arc are other singular invariant
objects. Three extensions seem natural. First, enlarging the class of
competitors to admissible trajectories of \eqref{eq:system} near the
abnormal cone, which requires combining the present construction with the
closed-measure approach to nonholonomic and control problems of \cite{CM}
and with the local time-minimal synthesis of \cite{BK93,BC03}. Second, near exceptional arcs the time-minimal synthesis is governed by the
interplay between hyperbolic and elliptic normal singular arcs and
bang-bang extremals \cite{BK93,BC03}. \Cref{thm:multiplier} shows that, when
\(\lambda_\Gamma\neq0\), the abnormal cone over \(\Gamma\) is the asymptotic limit
of normal singular arcs of both types; describing the conjugate and cut loci
of the normal extremals near the abnormal cone, and comparing them with the
Aubry set, is a natural problem, see \cite{BCP2026} for a related analysis of
abnormal geodesics and conjugate loci in a spherical setting. Third, higher-dimensional systems, where
the singular flow lives on a singular locus of higher dimension and genuine
recurrence (and thus a nontrivial Mather set) can appear beyond periodic
orbits; there the duality formula and the structure theorem of
\Cref{sec:AMstructures} apply without change.

The first of these directions raises a precise question. Freezing the
control bound at a given scale, \(\abs{u}\le\eta\abs{v}^3\), produces a
homogeneous control system whose projection to \(\Sph^2\) carries an
ergodic control problem: minimizing the asymptotic radial exponent of
admissible trajectories, that is, deciding asymptotic controllability to the
origin. Along singular arcs the running cost of this problem is \(\rho\), so
the structures of \Cref{sec:periodic} describe its singular competitors. A
formal application of the maximum principle along calibrated arcs suggests
that an abnormal component \(\Gamma\) can support an optimal measure of the
full problem only if the corresponding Hamiltonian level, which equals the
critical value, vanishes, that is, only if \(\ov\rho_\Gamma=0\) and the
abnormal cone over \(\Gamma\) is foliated by periodic orbits of
\eqref{eq:system} (\Cref{prop:lift}(c)). Establishing this rigidity
statement, and conversely identifying situations where neutral abnormal
components are selected at the threshold of asymptotic controllability, would
connect the singular Aubry--Mather structure with the stabilization of
homogeneous control systems; we leave it for future work.

% =================================================================


\begin{thebibliography}{99}

\bibitem{ABB20}
A.~Agrachev, D.~Barilari and U.~Boscain,
\newblock \emph{A Comprehensive Introduction to Sub-Riemannian Geometry},
\newblock Cambridge Studies in Advanced Mathematics \textbf{181}, Cambridge University Press, Cambridge, 2020.

\bibitem{AS04}
A.~A.~Agrachev and Yu.~L.~Sachkov,
\newblock \emph{Control Theory from the Geometric Viewpoint},
\newblock Encyclopaedia of Mathematical Sciences \textbf{87}, Springer, Berlin, 2004.

\bibitem{AgSa96}
A.~A.~Agrachev and A.~V.~Sarychev,
\newblock Abnormal sub-Riemannian geodesics: Morse index and rigidity,
\newblock \emph{Ann. Inst. H. Poincar{\'e} Anal. Non Lin{\'e}aire} \textbf{13} (1996), no.~6, 635--690.

\bibitem{Bernard}
P.~Bernard,
\newblock The dynamics of pseudographs in convex Hamiltonian systems,
\newblock \emph{J. Amer. Math. Soc.} \textbf{21} (2008), no.~3, 615--669.

\bibitem{Bernard08}
P.~Bernard,
\newblock Young measures, superposition and transport,
\newblock \emph{Indiana Univ. Math. J.} \textbf{57} (2008), no.~1, 247--275.

\bibitem{BC03}
B.~Bonnard and M.~Chyba,
\newblock \emph{Singular Trajectories and their Role in Control Theory},
\newblock Math{\'e}matiques \& Applications \textbf{40}, Springer, Berlin, 2003.

\bibitem{BCP2026}
B.~Bonnard, O.~Cots and Y.~Privat,
\newblock The Zermelo navigation problem on the 2-sphere of revolution:
an optimal control perspective with applications to micromagnetism,
\newblock \emph{Nonlinearity} \textbf{39} (2026), no.~1, 015006.

\bibitem{BK93}
B.~Bonnard and I.~Kupka,
\newblock Th{\'e}orie des singularit{\'e}s de l'application entr{\'e}e/sortie et
optimalit{\'e} des trajectoires singuli{\`e}res dans le probl{\`e}me du temps minimal,
\newblock \emph{Forum Math.} \textbf{5} (1993), no.~2, 111--159.

\bibitem{BMR}
B.~Bonnard, C.~Mendico and J.~Rouot,
\newblock Feedback equivalence and time minimal geodesics for controlled
Euler equation with a single actuator and attitude control,
\newblock HAL preprint, 2025.

\bibitem{Bousch00}
T.~Bousch,
\newblock Le poisson n'a pas d'ar{\^e}tes,
\newblock \emph{Ann. Inst. H. Poincar{\'e} Probab. Statist.} \textbf{36} (2000), no.~4, 489--508.

\bibitem{CM}
P.~Cannarsa and C.~Mendico,
\newblock Aubry--Mather theory for optimal control systems with nonholonomic constraints,
\newblock \emph{Math. Ann.}, published online 2025.

\bibitem{CI99}
G.~Contreras and R.~Iturriaga,
\newblock \emph{Global Minimizers of Autonomous Lagrangians},
\newblock 22\textsuperscript{o} Col{\'o}quio Brasileiro de Matem{\'a}tica, IMPA, Rio de Janeiro, 1999.

\bibitem{CLT01}
G.~Contreras, A.~O.~Lopes and Ph.~Thieullen,
\newblock Lyapunov minimizing measures for expanding maps of the circle,
\newblock \emph{Ergodic Theory Dynam. Systems} \textbf{21} (2001), no.~5, 1379--1409.

\bibitem{DLA06}
F.~Dumortier, J.~Llibre and J.~C.~Art{\'e}s,
\newblock \emph{Qualitative Theory of Planar Differential Systems},
\newblock Universitext, Springer, Berlin, 2006.

\bibitem{Evans}
L.~C.~Evans,
\newblock Weak KAM theory and partial differential equations,
\newblock in \emph{Calculus of Variations and Nonlinear Partial Differential Equations},
\newblock Lecture Notes in Mathematics \textbf{1927}, Springer, Berlin, 2008, pp.~123--154.

\bibitem{Fathi}
A.~Fathi,
\newblock \emph{The Weak KAM Theorem in Lagrangian Dynamics},
\newblock Cambridge Studies in Advanced Mathematics \textbf{88},
\newblock Cambridge University Press, Cambridge, 2008.

\bibitem{Foote84}
R.~L.~Foote,
\newblock Regularity of the distance function,
\newblock \emph{Proc. Amer. Math. Soc.} \textbf{92} (1984), no.~1, 153--155.

\bibitem{FS04}
A.~Fathi and A.~Siconolfi,
\newblock Existence of \(C^1\) critical subsolutions of the Hamilton--Jacobi equation,
\newblock \emph{Invent. Math.} \textbf{155} (2004), no.~2, 363--388.

\bibitem{GH55}
W.~H.~Gottschalk and G.~A.~Hedlund,
\newblock \emph{Topological Dynamics},
\newblock American Mathematical Society Colloquium Publications \textbf{36}, AMS, Providence, RI, 1955.

\bibitem{Harnack}
A.~Harnack,
\newblock {\"U}ber die Vieltheiligkeit der ebenen algebraischen Curven,
\newblock \emph{Math. Ann.} \textbf{10} (1876), 189--198.

\bibitem{Hughes}
P.~C.~Hughes,
\newblock \emph{Spacecraft Attitude Dynamics},
\newblock Dover Publications, Mineola, NY, 2004.
\newblock Reprint of the 1986 Wiley edition.

\bibitem{Jenkinson}
O.~Jenkinson,
\newblock Ergodic optimization in dynamical systems,
\newblock \emph{Ergodic Theory Dynam. Systems} \textbf{39} (2019), no.~10, 2593--2618.

\bibitem{Mane97}
R.~Ma{\~n}{\'e},
\newblock Lagrangian flows: the dynamics of globally minimizing orbits,
\newblock \emph{Bol. Soc. Brasil. Mat. (N.S.)} \textbf{28} (1997), no.~2, 141--153.

\bibitem{MarkleyCrassidis}
F.~L.~Markley and J.~L.~Crassidis,
\newblock \emph{Fundamentals of Spacecraft Attitude Determination and Control},
\newblock Space Technology Library \textbf{33},
\newblock Springer, New York, 2014.

\bibitem{Mather1991}
J.~N.~Mather,
\newblock Action minimizing invariant measures for positive definite Lagrangian systems,
\newblock \emph{Math. Z.} \textbf{207} (1991), no.~2, 169--207.

\bibitem{Mather1993}
J.~N.~Mather,
\newblock Variational construction of connecting orbits,
\newblock \emph{Ann. Inst. Fourier (Grenoble)} \textbf{43} (1993), no.~5, 1349--1386.

\bibitem{Montgomery94}
R.~Montgomery,
\newblock Abnormal minimizers,
\newblock \emph{SIAM J. Control Optim.} \textbf{32} (1994), no.~6, 1605--1620.

\bibitem{Sorrentino}
A.~Sorrentino,
\newblock \emph{Action-minimizing Methods in Hamiltonian Dynamics: An Introduction to Aubry--Mather Theory},
\newblock Mathematical Notes \textbf{50},
\newblock Princeton University Press, Princeton, NJ, 2015.

\bibitem{Wertz}
J.~R.~Wertz (ed.),
\newblock \emph{Spacecraft Attitude Determination and Control},
\newblock Astrophysics and Space Science Library \textbf{73},
\newblock D.~Reidel Publishing Company, Dordrecht, 1978.

\bibitem{Wilson}
G.~Wilson,
\newblock Hilbert's sixteenth problem,
\newblock \emph{Topology} \textbf{17} (1978), no.~1, 53--73.

\end{thebibliography}
\end{document}